\documentclass[a4paper,american,reqno]{amsart}

\newif\ifSubmission \Submissiontrue

\usepackage{babel}
\usepackage[utf8]{inputenc}
\usepackage[T1]{fontenc}
\usepackage[binary-units=true]{siunitx}
\usepackage{booktabs}
\usepackage{xspace}
\usepackage{accents}
\usepackage{algorithm}
\usepackage[noend]{algorithmic}
\usepackage{todonotes}
\usepackage{csquotes}
\usepackage{enumitem}
\setlist[enumerate,1]{label=(\roman*), left=-0.5em}
\usepackage{amssymb,amsmath,amsfonts,amsthm}
\allowdisplaybreaks
\usepackage{microtype}
\usepackage{pgfplots}
\usetikzlibrary{pgfplots.groupplots}
\usepackage[style = authoryear-comp,
            maxbibnames = 100,
            maxcitenames = 2,
            giveninits = true,
            uniquename = init,
            isbn = false,
            dashed = false,
            backend = bibtex]{biblatex}
\usepackage[colorlinks,
            citecolor=blue,
            urlcolor=blue,
            linkcolor=blue]{hyperref} 

\makeatletter
\patchcmd{\@settitle}{\uppercasenonmath\@title}{\scshape\large}{}{}
\patchcmd{\@setauthors}{\MakeUppercase}{\scshape\normalsize}{}{}
\makeatother

\vfuzz \hfuzz
\makeatletter
\@namedef{subjclassname@2020}{%
  \textup{2020} Mathematics Subject Classification}
\makeatother

\newcommand{\field}{\mathbb}

\newcommand{\reals}{\field{R}}
\newcommand{\rationals}{\field{Q}}

\newcommand{\naturals}{\field{N}}

\newcommand{\R}{\reals}
\newcommand{\Q}{\rationals}

\newcommand{\N}{\naturals}

\newcommand{\st}{\text{s.t.}}

\newcommand{\defset}[3][\defsep]{\set{#2#1#3}}
\newcommand{\Defset}[3][\defsep]{\Set{#2#1#3}}
\newcommand{\set}[1]{\{#1\}}
\newcommand{\Set}[1]{\left\{#1\right\}}

\newcommand{\Abs}[1]{\left\lvert#1\right\rvert}

\DeclareMathOperator*{\argmin}{arg\,min}
\DeclareMathOperator*{\argmax}{arg\,max}

\newcommand{\abbr}[1][abbrev]{#1.\xspace}

\newcommand{\eg}{\abbr[e.g]}

\newcommand{\ie}{\abbr[i.e]}

\newcommand{\Wlog}{\abbr[w.l.o.g]}

\newcommand{\define}{\mathrel{{\mathop:}{=}}}
\newcommand{\enifed}{\mathrel{{=}{\mathop:}}}

\newtheorem{theorem}{Theorem}

\newtheorem{proposition}{Proposition}
\newtheorem*{proposition*}{Proposition}

\newtheorem{lemma}{Lemma}

\newcommand{\scenarios}{S}

\newcommand{\initCost}{f}
\newcommand{\cost}{c}
\newcommand{\costFun}[2]{\cost_{#1}(#2)}
\newcommand{\nomCostFun}[2]{\bar{\cost}_{#1}(#2)}
\newcommand{\compCostFun}[3]{\tilde{\cost}^{\,#3}_{#1}(#2)}
\newcommand{\devCostFun}[2]{\Delta \cost_{#1}(#2)}
\newcommand{\scenarioCostFun}[3]{\cost^{#3}_{#1}(#2)}
\newcommand{\nomCost}{\bar{\cost}}
\newcommand{\devCost}{\Delta \cost}
\newcommand{\scenarioCost}[1]{\cost^{#1}}
\newcommand{\uncertaintySet}{\mathcal{U}}
\newcommand{\newUncertaintySet}{\mathcal{V}}
\newcommand{\negSet}[1]{N(#1)}
\newcommand{\posSet}[1]{P(#1)}
\newcommand{\leSet}[1]{L(#1)}
\newcommand{\geqSet}[1]{G(#1)}

\definecolor{my-red}{HTML}{d7191c}
\definecolor{my-blue}{HTML}{2b83ba}
\definecolor{my-yellow}{HTML}{fdae61}
\definecolor{my-green}{HTML}{abdda4}
\definecolor{my-purple}{HTML}{af8dc3}

\def\cross{\textcolor{my-red}{\textsf{X}}}
\renewcommand{\check}{\textcolor{my-green!70!black}{\checkmark}}

\bibliography{k-delete-recoverable-robust}

\begin{document}

\title[Exact Methods for $k$-Delete Recoverable Robust Problems]{Exact
  Methods for Solving $k$-Delete Recoverable\\Robust $0$-$1$ Problems Under
  Budgeted Uncertainty}

\author[Y. Beck, C. Büsing, I. Ljubi\'c]{Yasmine Beck, Christina Büsing, Ivana
  Ljubi\'c}

\address[Y. Beck]{%
  Eindhoven University of Technology,
  Department of Industrial Engineering and Innovation Sciences,
  PO Box 513,
  5600 MB Eindhoven,
  the Netherlands}%
\email{y.beck@tue.nl}

\address[I.\ Ljubi\'c]{%
  ESSEC Business School,
  Department of Information Systems, Data Analytics and Operations,
  3 Avenue Bernard Hirsch,
  95021 Cergy-Pontoise Cedex,
  France}%
\email{ljubic@essec.edu}

\address[C.\ Büsing]{%
  Teaching and Research Area Combinatorial Optimization,
  RWTH Aachen University,
  Ahornstra\ss e 55,
  52074 Aachen,
  Germany}%
\email{buesing@combi.rwth-aachen.de}

\date{\today}

\begin{abstract}
  We study the~$k$-delete recoverable robust $0$--$1$ problem in which a
decision-maker solves a combinatorial optimization problem subject to objective
uncertainty. The model follows a two-stage robust setup. The
decision-maker first commits to an initial plan and may then revoke up to~$k$
components of this decision after the uncertainty is revealed.
The underlying uncertainty is modeled using a budgeted uncertainty set so
that the decision-maker only hedges against a limited number of deviations in
the uncertain parameters.
We present four reformulations of the~$k$-delete recoverable
robust problem, which can be tackled using (i)~general-purpose mixed-integer
linear programming solvers,
(ii) branch-and-cut methods, or (iii)~column-and-constraint generation
algorithms.
For each formulation, we identify suitable solution methods and prove their
correctness.
Overall, we present eight approaches to solve the~$k$-delete recoverable
robust problem, which we assess and compare in an extensive computational study
on instances of the assignment problem and the single-source capacitated
facility location problem.


\end{abstract}

\keywords{Combinatorial optimization,
Mixed-integer programming,
Robust optimization,
Recoverable robustness,
Budgeted uncertainty%
%
%
}
\subjclass[2020]{90C11, 90C27, 90C57, 90C70%
%
%
}

\maketitle

\section{Introduction}
\label{sec:introduction}

Decision-makers are frequently forced to commit to choices without being able
to fully anticipate their consequences as real-world data often contains
inherent uncertainties such as measurement errors or incomplete information.
However, even small perturbations in the data can render a decision sub-optimal
or infeasible for the problem at hand; see, e.g., the case study in
\textcite{Ben-Tal_et_al:2009} for an illustrative example. It is thus
essential to explicitly account for uncertainty when designing decision-support
tools.

One approach to deal with uncertainties in mathematical optimization is
to exploit techniques from robust optimization
\parencite{Bertsimas_et_al:2011,Ben-Tal_et_al:2009,Soyster:1973}.
Classic robust models aim for solutions that remain feasible for all possible
realizations of the uncertain parameters within a prescribed uncertainty set.
However, because feasibility must be ensured even for extreme and often
unlikely realizations of uncertainty and performance is evaluated with
respect to the worst-case outcome, such models may discard decisions that
would perform well in most realistic situations.
As a result, classic robust solutions are often overly conservative.
This drawback has motivated a growing interest in more flexible robust
approaches for dealing with uncertainties, including adjustable robustness
\parencite{Ben-Tal_et_al:2004}, light robustness
\parencite{Fischetti_Monaci:2009}, distributional robustness
\parencite{Goh-Sim:2010,Wiesemann_et_al:2014}, $\Gamma$-robustness
\parencite{Bertsimas_Sim:2003,Bertsimas_Sim:2004,Sim:2004}, and recoverable
robustness \parencite{Liebchen_et_al:2009}. The latter two are also at the
core of this paper.
In the~$\Gamma$-robust approach, also known as the budgeted uncertainty
model, the assumption that all uncertain parameters realize in a worst-case
sense is relaxed.
Instead, the decision-maker only hedges against up to~$\Gamma$
deviations in the uncertain parameters that adversely affect the outcome.
Recoverable robustness, in contrast, involves a two-stage robust framework.
In the first stage, the decision-maker takes some decisions in a here-and-now
fashion, i.e., before the uncertainty is revealed. In the second stage, a
limited number of recovery actions can then be applied to adjust the
first-stage decisions once the actual values of the uncertain parameters become
known. These recovery actions are taken in a wait-and-see manner and may depend
on the first-stage decisions.

In this paper, we introduce the~$k$-delete recoverable robust problem in
which we combine recoverable robustness with a budgeted uncertainty modeling.
To this end, we focus on combinatorial optimization problems in which the
objective function coefficients are subject to uncertainty.
We emphasize that we do not make assumptions about the specific structure of
the underlying $0$--$1$ problem. Hence, our setting covers a broad range of
well-known combinatorial optimization problems under objective uncertainty,
including the assignment problem, the facility location problem, the minimum
spanning tree problem, or the knapsack problem, all of which are of practical
interest. In particular, our modeling framework is relevant for these problems
as their underlying structure---for example, a given graph---typically remains
fixed, while only some cost coefficients are affected by uncertainty.
In our model, the available recovery actions consist of revoking, i.e.,
deleting, up to~$k$~decisions taken in the first stage.
Incorporating the possibility of such recovery actions in the planning
phase supports more flexible and informed decision-making. In particular, it
can substantially reduce costs compared to withdrawing commitments
after solving a deterministic model.
%
%
To illustrate the practical relevance of this framework, let us brief\/ly
discuss two representative applications from telecommunications and preventive
maintenance.

\subsubsection*{Telecommunications}
Consider a telecommunication company that needs to decide on the optimal
deployment of fiber-to-the-home technology for a given set of customers; see,
e.g., \textcite{Groetschel_et_al:2014} for a primer on the optimization
and integer programming aspects of optical access network planning.
The company has a number of available facilities (e.g., multiplexers or
switchers) from which fiber-optic cables must be provided to end-customers.
In the first stage, the company needs to commit to a service provision plan by
deciding on an optimal assignment between facilities and customers.
In the future, the company intends to upgrade the service to a next-generation
technology. The corresponding upgrade costs are determined by external
suppliers of the underlying equipment and are, thus, unknown at the current
stage.
Once the actual upgrade prices become available, the company can decide
which customers to upgrade.
However, only a limited number of customers may receive the upgrade, whereas
the remaining ones are left with the old technology.
Business solutions in which some of the customers are left unserved or not
upgraded to the newest technology are quite common in practice
as fully satisfying all customer demand often exceeds the company's available
budget.

\subsubsection*{Preventive Maintenance}
Preventive maintenance of machinery and equipment is essential to ensure safe,
reliable, and uninterrupted operations across various domains; see, e.g., the
surveys in \textcite{Vasili_et_al:2011,de-Jonge_Scarf:2020}.
Maintenance schedules typically need to be created well in advance, which is
why the associated costs---such as labor, spare parts, or downtime---are often
uncertain at the planning stage.
As these uncertainties realize over time, some maintenance tasks may become too
expensive to justify or less critical than initially anticipated. For instance,
it may become economically unviable to maintain a machine with a low
failure-risk or a component may be in better condition than expected, e.g.,
because it has just been repaired or replaced.
In practice, fully rescheduling all maintenance tasks may be infeasible due to
operational, technical, or regulatory constraints.
Nevertheless, a company may adjust the plan once the actual maintenance costs
are known, but only a limited number of tasks may be postponed or canceled.

\subsection*{Related Literature}

Since its introduction in the context of train scheduling by
\textcite{Liebchen_et_al:2009}, recoverable robustness has gained increasing
attention over the last two decades. Early contributions such as
\textcite{Cicerone_et_al:2009,Buesing:2009,Cacchiani_et_al:2008} already
highlight its relevance for public transportation applications and propose
first algorithmic approaches for solving recoverable robust optimization
problems.
In the following years, recoverable robustness has been studied for various
combinatorial optimization problems and uncertainty models. In particular,
column generation methods for problems with discrete scenario sets have been
proposed in \textcite{van-den-Akker_et_al:2016,Bouman_et_al:2011}, with
applications to the size robust knapsack problem and the demand robust shortest
path problem.
Further applications of recoverable robustness include facility location and
allocation problems \parencite{Alvarez-Miranda_et_al:2015b} as well as
telecommunication network design \parencite{Alvarez-Miranda_et_al:2015}.
Moreover, knapsack problems have emerged as one of the most frequently studied
settings for recoverable robustness; see, e.g.,
\textcite{Buesing_et_al:2011a,Buesing_et_al:2011b,Buesing_et_al:2019}.
In \textcite{Buesing_et_al:2011a}, the authors consider uncertainties in both the
profits and the weights and allow recovery actions that include deleting up
to~$k$ items from and adding up to~$\ell$ items to the first-stage decision.
The uncertainties are modeled using a discrete scenario set. A similar setting
is studied in \textcite{Buesing_et_al:2011b,Buesing_et_al:2019}, 
where uncertain weights are handled using a budgeted uncertainty model and
the deletion of up to~$k$ items is allowed as recovery action.
A recovery structure similar to that in \textcite{Buesing_et_al:2011a} is also
considered in \textcite{Hommelsheim_et_al:2023}, where the authors study
recoverable robust combinatorial optimization problems for problem classes
satisfying the hereditary property.
More recently, complexity aspects of recoverable robust problems under discrete
budgeted uncertainty have been addressed in \textcite{Gruene_Wulf:2024} for
various underlying combinatorial optimization problems.

In this paper, we pursue similar ideas as in
\textcite{Buesing_et_al:2011b,Buesing_et_al:2019} so that the considered
recovery actions include deleting up to~$k$ components of the first-stage
decision.
Our contributions complement the aforementioned literature in that we do not
make assumptions about the structure of the underlying combinatorial
optimization problem in our recoverable robust setup. Hence, our
approaches can be applied to any $0$--$1$ problem subject to objective
uncertainty. Moreover, whereas much of the existing literature focuses on
structural properties and complexity results of specific problem classes, our
aim is to present different algorithmic approaches to solve~$k$-delete
recoverable robust combinatorial problems to global optimality.
To this end, we propose four equivalent reformulations of the problem, which
can be tackled using (i)~general-purpose mixed-integer linear programming
solvers, (ii)~branch-and-cut methods, or (iii)~column-and-constraint generation
algorithms.
Overall, we present eight approaches to solve the~$k$-delete recoverable robust
problem and assess and compare their performance in an extensive computational
study.

\subsection*{Outline}

The remainder of this paper is organized as follows.
In Section~\ref{sec:problem-statement}, we present the overall problem
statement and introduce the~$k$-delete recoverable robust problem.
In Section~\ref{sec:reformulations}, we derive four equivalent reformulations
of the problem: a scenario-based, a projection-based, an extended, and a
compact reformulation. For each formulation, we discuss exact solution
approaches in Section~\ref{sec:solution-approaches}.
In Section~\ref{sec:computational-results}, we conduct an extensive
computational study on instances of the assignment problem and the
single-source capacitated facility location problem to assess the
performance of the proposed solution methods. Finally, we derive conclusions in
Section~\ref{sec:conclusion}.


\section{Problem Statement and Connections to Other Problem Classes}
\label{sec:problem-statement}

We study the~$k$-delete recoverable robust $0$--$1$ problem
under objective uncertainty.
In this setting, a decision-maker determines an initial
plan~\mbox{$x \in X \subseteq \set{0,1}^n$} before the uncertainty is revealed,
where the vector~$x$ represents the first-stage decision.
Throughout this paper, we assume that~$X \neq \emptyset$ holds. Moreover,
let~$\initCost \in \R^n_{\geq 0}$ denote the fixed costs associated with the
first-stage decision.
Then, after the first-stage decision has been made, a
scenario~$s \in \scenarios$ from a compact set~$\scenarios$ of possible
uncertainty realizations is revealed, which determines the corresponding
objective function coefficients (or scenario costs)
$\scenarioCost{s} \in \R^n_{\geq 0}$ for recovery actions.
The decision-maker may then apply recovery actions and revoke up
to~$k \in \set{0,\ldots,n}$ components of the first-stage decision.
The aim is to select a first-stage decision that minimizes the total
cost, consisting of the initial cost of the plan and the
worst-case cost incurred for recovery actions.
For given~$x \in X$ and~$s \in \scenarios$, the total cost of recovery
is defined as
\begin{equation}
  \label{eq:recourse-problem}
  R(x,s) \define \min_{y} \Defset{\sum_{i=1}^n c^s_iy_i}{y \leq x,\,
    \sum_{i=1}^n y_i \geq \sum_{i=1}^n x_i - k,\, y \in \Set{0,1}^n}.
\end{equation}
Overall, the~$k$-delete recoverable robust problem can thus be stated as
\begin{equation}
  \label{eq:recoverable-robust-problem}
  \min_x \quad \initCost^\top x + \max_{s \in \scenarios} \Set{R(x,s)}
  \quad \st \quad x \in X \subseteq \Set{0,1}^n.
\end{equation}
Note that if~$\Abs{\scenarios} = 1$, the~$k$-delete recoverable
robust problem~\eqref{eq:recoverable-robust-problem} reduces to a deterministic
two-stage problem.
The recovery actions~$y$ correspond to the second-stage decisions and allow the
decision-maker to adjust the initial plan by deleting at most~$k$ components
of~$x$, where the parameter~$k$ is used to control the degree of
flexibility. Setting~$k = 0$ yields a classic robust problem without recovery
actions, whereas~$k = n$ allows all first-stage decisions to be revised.
In what follows, we use the abbreviation~$N \define \set{1,\ldots,n}$ so that
Problem~\eqref{eq:recourse-problem} can equivalently be stated as
\begin{equation*}
  R(x,s) = \sum_{i=1}^n c^s_ix_i - \max_{\Defset{K \subseteq N}{\Abs{K} \leq
      k}} \sum_{i \in K} c^s_ix_i
\end{equation*}
for given~$x \in X$ and~$s \in \scenarios$.

We point out that, in Problem~\eqref{eq:recoverable-robust-problem}, all
decisions can be recovered in principle.
This is done for the ease of presentation but we acknowledge that
more general formulations in which different parts of the decision are
treated independently in terms of uncertainties and allowed recovery actions
are possible as well.
Such generalizations can, \eg, be obtained by partitioning the index set~$N$
into recoverable and non-recoverable decisions.
For instance, this is suitable in the context of facility location
problems. An opened facility may not be closed again, \ie,
no recovery action is allowed, but the connection to an assigned customer may
be revoked.
Furthermore, cost uncertainties for the facility locations, if present at all,
can behave quite differently than uncertainties regarding the assignment
costs.

We further mention that our modeling of the~$k$-delete recoverable robust
problem allows recovered solutions to violate some constraints of the
original model. For instance, some customers may remain unserved in a facility
location problem, which can be a common situation in practice.
However, if the underlying problem satisfies the downward monotonicity (or
hereditary) property, feasibility after recovery is guaranteed.
Examples of such problems include the knapsack problem, the maximum clique
problem, and other subset-selection problems in which any subset of a feasible
solution remains feasible.


Let us now elaborate on the scenario set and the scenario costs considered in
this paper.
In robust optimization, commonly studied scenario sets include interval-,
discrete-, ellipsoidal-, or $\Gamma$-scenario sets. In particular,
the~$\Gamma$-scenario approach
\parencite{Bertsimas_Sim:2003,Bertsimas_Sim:2004,Sim:2004} is among the most
prominent and frequently used methods to address uncertainties in robust
optimization. Hence, and as it seems unlikely that all uncertain costs realize
in a worst-case sense, we also adopt such a budgeted uncertainty modeling in
this paper.
To this end, we assume that the recovery costs are
uncertain but known to take values in the uncertainty set
\begin{equation*}
  \mathcal{U} = \mathcal{U}_1 \times \dotsb \times \mathcal{U}_n
  \quad \text{with} \quad
  \mathcal{U}_i \define [\nomCost_i, \nomCost_i + \devCost_i]
  \quad \text{for all } i \in N.
\end{equation*}
Here, $\nomCost_i \geq 0$ is the nominal value of the cost of
recovery action~$i \in N$ and~$\devCost_i \geq 0 $ is the maximum deviation
from the nominal value, i.e., the maximum cost increase.
Following the~$\Gamma$-robust approach, the decision-maker now only
hedges against a subset of at
most~$\Gamma \in \set{0, \ldots, n}$ deviations in the cost coefficients that
adversely affect the solution to the problem at hand.
Here, the parameter~$\Gamma$ is used to control the
decision-maker's level of conservatism.
In particular, our modeling captures the nominal ($\Gamma = 0$) and the
strictly robust formulation ($\Gamma = n$) as special cases.
For a given~$\Gamma$, we define
\begin{equation*}
  \uncertaintySet_\Gamma\
  \define \Defset{u \in \Set{0,1}^n}{\sum_{i=1}^n u_i \leq \Gamma}
  = \Set{u^1, u^2, \ldots, u^M}
\end{equation*}
so that the discrete scenario set under budgeted uncertainty can be stated as
\begin{equation*}
  \scenarios = \Defset{s \in \Set{1,2,\ldots, M}}{
    \scenarioCost{s}_i \define \nomCost_i + \devCost_i u^s_i,\, i \in N,\,
    u^s \in \uncertaintySet_\Gamma}.
\end{equation*}
The number of scenarios~$\Abs{S} = M$ corresponds to the number of binary
vectors of length~$n$ for which at most~$\Gamma$ many entries are ones. Hence,
we have~$M \in O(n^\Gamma)$.
\subsection{Connections to Decision-Dependent Robust and Bilevel Optimization}

In Problem~\eqref{eq:recoverable-robust-problem}, a decision-maker may delete
up to~$k$ first-stage decisions after the uncertainty is revealed. Naturally,
such deletions are only relevant for decisions that are indeed taken in the
first stage and for which the cost increases. Consequently, the recovery
actions are inherently dependent on the first-stage decision. This establishes
a close relationship between~$k$-delete recoverable robustness and robust
optimization with decision-dependent uncertainty sets; see, e.g.,
\textcite{Nohadani_Sharma:2018,Poss:2013,Poss:2014,Arslan_Poss:2024} for
influential works in this area.
In particular, under certain assumptions, the~$k$-delete recoverable robust
problem~\eqref{eq:recoverable-robust-problem} can be explicitly formulated as a
decision-dependent robust optimization problem, as we illustrate in the
following.

\begin{proposition}
  \label{prop:relation-to-ddrop}
  For given~$\Gamma,k \in \set{0,\ldots,n}$, suppose that
  $$\max_{x \in X} \Set{ \sum_{i=1}^n x_i } \enifed \bar n \leq \Gamma + k$$
  holds.
  Then, the~$k$-delete recoverable robust
  problem~\eqref{eq:recoverable-robust-problem} can be solved as the
  robust combinatorial optimization problem
  \begin{equation}
    \label{eq:ddrop}
    \min_{x \in X} \Set{ \sum_{i=1}^n \initCost_ix_i
      + \max_{u \in \mathcal{U}(x)} \Set{\sum_{i=1}^n (\nomCost_i + \devCost_i)u_i}}
  \end{equation}
  with the discrete decision-dependent uncertainty set
  \begin{equation*}
    \mathcal{U}(x) = \Defset{u \in \set{0,1}^n}{x \geq u,\, \sum_{i=1}^n u_i
      \leq \sum_{i=1}^n x_i - k}.
  \end{equation*}
\end{proposition}
\begin{proof}
  The~$k$-delete recoverable robust problem under budgeted uncertainty can
  equivalently be stated as
  \begin{equation*}
    \min_{x \in X} \Set{\initCost^\top x + \max_{u \in \uncertaintySet_\Gamma}
    \Set{ \min_{y \in \set{0,1}^n} \Defset{\sum_{i=1}^n \left( \nomCost_i +
        \devCost_iu_i \right) y_i}{y \leq x,\,
      \sum_{i=1}^n y_i \geq \sum_{i=1}^n x_i - k}}}.
  \end{equation*}
  Let~$x^*$ be an optimal first-stage decision to this problem and let~$y^*$ be
  its associated optimal recovery action under a worst-case uncertainty
  realization. Because all cost coefficients are non-negative, it is
  optimal to delete as many components as possible, i.e.,
  $\sum_{i=1}^n y^*_i = \sum_{i=1}^n x^*_i - k$ holds.
  We now show that there exists a worst-case uncertainty
  realization~$u^* \in \uncertaintySet_\Gamma$ with~$u^* = y^*$.
  On the one hand, if~$y^*_i = 0$ holds for some~$i \in N$, setting~$u^*_i = 1$
  does not increase the objective value. Hence, there must exist a worst-case
  realization~$u^*$ with~$u^*_i = 0$ for all~$i \in N$ satisfying~$y^*_i = 0$.
  On the other hand, if~$y^* = 1$ holds for some~$i \in N$,
  setting~$u^*_i = 1$ leads to an objective value that is at least as large as
  in the case with~$u^*_i = 0$ because~$\devCost_i \geq 0$ holds.
  Hence, we set~$u^*_i = 1$ for all~$i \in N$ with~$y^*_i = 1$.
  By construction, this yields~$u^* \in \set{0,1}^n$ and
  \begin{equation*}
    \sum_{i=1}^n u^*_i = \sum_{i=1}^n y^*_i = \sum_{i=1}^n x^*_i - k \leq \bar
    n - k \leq \Gamma,
  \end{equation*}
  i.e., $u^* \in \uncertaintySet_\Gamma$.
  In particular, each~$i \in N$ is associated with either no recovery
  cost (if~$u^*_i = 0$) or the full recovery cost~$\nomCost_i + \devCost_i$
  (if~$u^*_i = 1$).
  Defining the decision-dependent uncertainty set accordingly, we conclude
  that (i)~$x^*$ is feasible for Problem~\eqref{eq:ddrop},
  (ii)~$u^* \in \mathcal{U}(x^*)$ is a worst-case uncertainty realization
  for~$x^*$, and (iii)~the~$k$-delete recoverable robust problem and
  Problem~\eqref{eq:ddrop} attain the same objective value.

  Let now~$x^* \in X$ be an optimal solution to Problem~\eqref{eq:ddrop} and
  let~$u^* \in \mathcal{U}(x^*)$ be its associated worst-case uncertainty
  realization.
  Because~$u^*$ is optimal for the inner maximization problem
  in~\eqref{eq:ddrop}, we have~$\sum_{i=1}^n u^*_i = \sum_{i=1}^n x^*_i - k$
  and~$x^* \geq u^*$.
  Setting~$y^* = u^*$ thus yields a feasible recovery action for the given
  first-stage decision~$x^*$, implying that
  Problems~\eqref{eq:recoverable-robust-problem} and~\eqref{eq:ddrop} admit the
  same objective value.
  Overall, this shows that both problems are equivalent whenever~$\bar n - k
  \leq \Gamma$.
\end{proof}

The connection between Problem~\eqref{eq:recoverable-robust-problem}
and decision-dependent robust optimization established in
Proposition~\ref{prop:relation-to-ddrop}
also indicates a close relationship between~$k$-delete recoverable robustness
and bilevel optimization, as highlighted by recent results in
\textcite{Lefebvre_et_al:2025,Goerigk_et_al:2025}.


\section{Reformulations}
\label{sec:reformulations}

Recoverable robust combinatorial optimization problems are notoriously hard to
solve; see, e.g., \textcite{Gruene_Wulf:2024} for recent complexity results. In
addition, the nonlinear objective of the~$k$-delete recoverable robust
problem~\eqref{eq:recoverable-robust-problem} introduces further computational
challenges, motivating the need for effective solution approaches. In this
section, we present four reformulations
of~\eqref{eq:recoverable-robust-problem}---a scenario-based, a
projection-based, an extended, and a compact reformulation---for which we
develop exact solution methods in Section~\ref{sec:solution-approaches}.

\subsection{Scenario-Based Formulation}
\label{sec:scenario-based-ref}

The first reformulation of Problem~\eqref{eq:recoverable-robust-problem} is
obtained by introducing binary variables~$y^s \in \set{0,1}^n$ and additional
constraints to model the recovery actions for each scenario~$s \in \scenarios$.
To hedge against the worst-case realization of the uncertainty, we
use an epigraph reformulation so that
Problem~\eqref{eq:recoverable-robust-problem} can
equivalently be re-stated as
\newpage
\begin{subequations}
  \label{eq:scenario-ref}
  \begin{align}
    \min_{x,y,\eta} \quad & \initCost^\top x + \eta
    \\
    \st \quad & \eta \geq \sum_{i=1}^n \scenarioCost{s}_i y^s_i,
                \quad s \in \scenarios,
                \label{eq:scenario-ref:worst-case}
    \\
    & \sum_{i=1}^n y^s_i \geq \sum_{i=1}^n x_i - k, \quad s \in \scenarios,
      \label{eq:scenario-ref:k-delete}
    \\
    & x \geq y^s, \quad s \in \scenarios,
      \label{eq:scenario-ref:recover}
    \\
    & x \in X,\, \eta \in \R_{\geq 0},\, y^s \in \Set{0,1}^n, \quad s \in
      \scenarios.
  \end{align}
\end{subequations}
Here and in what follows, we abbreviate $v = (v^a)_{a \in A}$ with $A$ being a
discrete finite set and~$v^a$, $a \in A$, being a vector of appropriate
dimension.
The constraints in~\eqref{eq:scenario-ref:worst-case} capture that the
decision-maker hedges against the worst-possible realization of the
uncertainty, whereas~\eqref{eq:scenario-ref:k-delete} guarantees that at
most~$k$ components of the first-stage decision~$x \in X$ are deleted in each
scenario~$s \in \scenarios$.
The constraints in~\eqref{eq:scenario-ref:recover} model that only those
elements contained in a fist-stage decision can be deleted.

A feature of the scenario-based formulation~\eqref{eq:scenario-ref} is that a
solution also specifies the recovery action to implement for each
scenario.
However, this leads to a considerably large model in general, as the number of
additional variables and constraints grows as~$O(n^{\Gamma+1})$.
In general, it is thus not possible to apply general-purpose MILP solvers
directly to solve Problem~\eqref{eq:scenario-ref}.

\subsection{Projection-Based Formulation}
\label{sec:proj-ref}

The second reformulation of Problem~\eqref{eq:recoverable-robust-problem} is
obtained by projecting the recovery actions onto the first-stage decisions and
considering a new scenario set. For this purpose, we build on the fact
that Problem~\eqref{eq:recoverable-robust-problem} is equivalent to
\begin{equation*}
  \min_x \quad \initCost^\top x + \beta(x) \quad \st \quad x \in X
\end{equation*}
with
\begin{equation*}
  \beta(x) = \max_u \Defset{\sum_{i=1}^n \nomCost_ix_i + \sum_{i=1}^n
    \devCost_ix_iu_i - \gamma(x,u)}{u \in \uncertaintySet_\Gamma}
\end{equation*}
and
\begin{equation*}
  \gamma(x,u) = \max_y \Defset{\sum_{i=1}^n \left( \nomCost_ix_i +
      \devCost_ix_iu_i
    \right)y_i}{\sum_{i=1}^n y_i \leq k,\, y \in \Set{0,1}^n}.
\end{equation*}
In addition, we introduce the following notation, which we use throughout the
remainder of this paper. We define
\begin{equation*}
  \newUncertaintySet \define \Set{0} \cup \Defset{\nomCost_i}{i \in N} \cup
  \Defset{\nomCost_i + \devCost_i}{i \in N}.
\end{equation*}
Moreover, for each~$\nu \in \newUncertaintySet$, we define the cost functions
\begin{equation*}
  \costFun{i}{\nu} \define \min \Set{\devCost_i, \nu - \nomCost_i},
  \quad i \in N,
\end{equation*}
and set
\begin{equation*}
  \nomCostFun{i}{\nu} \define
  \begin{cases}
    \nu, & \text{if } \nu < \nomCost_i, \\
    \nomCost_i, & \text{if } \nu \geq \nomCost_i, \\
  \end{cases}
  \qquad \text{and} \qquad
  \devCostFun{i}{\nu} \define
  \begin{cases}
    0, & \text{if } \nu < \nomCost_i, \\
    \costFun{i}{\nu}, & \text{if } \nu \geq \nomCost_i. \\
  \end{cases}
\end{equation*}
For notational convenience, we further define
\begin{equation*}
  \negSet{\nu} \define \Defset{i \in N}{\costFun{i}{\nu} < 0}
  \quad \text{and} \quad
  \posSet{\nu} \define \Defset{i \in N}{\costFun{i}{\nu} \geq 0},
\end{equation*}
as well as
\begin{equation*}
  \leSet{\nu} \define \Defset{i \in N}{\costFun{i}{\nu} < \devCost_i}
  \quad \text{and} \quad
  \geqSet{\nu} \define \Defset{i \in N}{\costFun{i}{\nu} \geq \devCost_i}
\end{equation*}
for all~$\nu \in \newUncertaintySet$.
By definition, $\negSet{\nu} \subseteq \leSet{\nu}$, $\geqSet{\nu} \subseteq
\posSet{\nu}$, and~$N = \negSet{\nu} \cup \posSet{\nu}$ holds for
any~$\nu \geq 0$.

Before we can formally state the projection-based reformulation of
Problem~\eqref{eq:recoverable-robust-problem}, we need the following
intermediate results.
We start with two technical lemmas, whose proofs are deferred to
Appendix~\ref{sec:proofs} to keep this section self-contained.

\begin{lemma}
  \label{lem:proj-ref}
  For arbitrarily given~$x \in X$, it holds
  \begin{equation}
    \label{eq:proj-ref-alpha}
    \begin{split}
      \beta(x) = 
      \max_{u \in \uncertaintySet_\Gamma,\nu \geq 0}
      \Set{\sum_{i=1}^n \nomCost_i x_i - k\nu
      + \sum_{i \in \negSet{\nu}} \costFun{i}{\nu} x_i
      + \sum_{i \in \posSet{\nu}} \costFun{i}{\nu} x_i u_i}.
    \end{split}
  \end{equation}
\end{lemma}

\begin{lemma}
  \label{lem:pos-neg-set}
  Let $\nu^* \geq 0$ be given arbitrarily and let
  \begin{equation*}
    \underline{\nu} \define \max \Defset{\nu}{\nu \in \newUncertaintySet,\, \nu
      \leq \nu^*} 
    \quad \text{and} \quad
    \bar{\nu} \define \min \Defset{\nu}{\nu \in \newUncertaintySet,\, \nu \geq \nu^*}.
  \end{equation*}
  Then, the following statements hold:
  \begin{enumerate}
  \item $\negSet{\bar{\nu}} \subseteq \negSet{\nu^*} \subseteq
    \negSet{\underline{\nu}}$ and
    $\posSet{\underline{\nu}} \subseteq \posSet{\nu^*} \subseteq
    \posSet{\bar{\nu}}$,
    \vspace*{0.25em}
  \item $\costFun{i}{\underline{\nu}} + (\nu^* - \underline{\nu}) = \costFun{i}{\nu^*}
    = \costFun{i}{\bar{\nu}} - (\bar{\nu} - \nu^*)$ 
    for all~$i \in \leSet{\nu^*}$,
    \vspace*{0.25em}
  \item $\costFun{i}{\underline{\nu}} = \costFun{i}{\nu^*} = \costFun{i}{\bar{\nu}}$
    for all~$i \in \geqSet{\nu^*}$,
    \vspace*{0.25em}
  \item $i \in \negSet{\nu^*} \cap \posSet{\bar{\nu}} \implies
    \costFun{i}{\bar{\nu}} = 0$,
    \vspace*{0.25em}
  \item $\nu^* \notin \newUncertaintySet \implies \posSet{\nu^*} \cap
    \negSet{\underline{\nu}} = \emptyset$.
  \end{enumerate}
\end{lemma}

We use Lemmas~\ref{lem:proj-ref} and~\ref{lem:pos-neg-set} to establish the
following intermediate result.

\begin{proposition}
  \label{prop:proj-ref}
  Problem~\eqref{eq:recoverable-robust-problem} can be solved as
  \begin{equation*}
    \min_{x \in X} \ \initCost^\top x
    + \max_{\nu \in \newUncertaintySet} \Set{\alpha(x,\nu) - k\nu},
  \end{equation*}
  where, for given~$x \in X$ and~$\nu \in \newUncertaintySet$, we have
  \begin{equation*}
    \alpha(x,\nu) = \max_{u}\Defset{\sum_{i=1}^n \nomCostFun{i}{\nu} x_i
      + \sum_{i=1}^n \devCostFun{i}{\nu} x_i u_i}{u \in
      \uncertaintySet_\Gamma}.
  \end{equation*}
\end{proposition}
\begin{proof}
  Let~$x \in X$ be given arbitrarily. Further, let~$(u^*,\nu^*)$ be an optimal
  solution to Problem~\eqref{eq:proj-ref-alpha}. We prove the claim by
  contradiction. To this end, suppose that~$\nu^* \notin \newUncertaintySet$
  holds and let~$\underline{\nu}$ and~$\bar{\nu}$ be defined as in
  Lemma~\ref{lem:pos-neg-set}. We distinguish two cases.
  First, suppose that
  \begin{equation*}
    \sum_{i \in \negSet{\nu^*}} x_i  + \sum_{i \in \posSet{\nu^*} \cap
      \leSet{\nu^*}} x_i u^*_i \leq k
  \end{equation*}
  holds.
  Then, by Lemmas~\ref{lem:proj-ref} and~\ref{lem:pos-neg-set}, we obtain
  \begin{align*}
    \beta(x)
    \overset{\text{Lemma~\ref{lem:proj-ref}}}{=}
    & \ \sum_{i=1}^n \nomCost_i x_i - k\nu^*
      + \sum_{i \in \negSet{\nu^*}} \costFun{i}{\nu^*} x_i
      + \sum_{i \in \posSet{\nu^*}} \costFun{i}{\nu^*} x_i u^*_i
    \\
    \overset{\phantom{\text{Lemma~\ref{lem:proj-ref}}}}{=}
    & \ \sum_{i=1}^n \nomCost_i x_i - k\nu^*
      + \sum_{i \in \negSet{\nu^*}} \costFun{i}{\nu^*} x_i
    \\
      & + \sum_{i \in \posSet{\nu^*} \cap \leSet{\nu^*}} \costFun{i}{\nu^*} x_i
        u^*_i
        + \sum_{i \in \posSet{\nu^*} \cap \geqSet{\nu^*}} \costFun{i}{\nu^*}
        x_i u^*_i
    \\
    \overset{\substack{\text{Lemma~\ref{lem:pos-neg-set}}\\\text{(ii),(iii)}}}{=}
    & \ \sum_{i=1}^n \nomCost_i x_i - k\underline{\nu} - k(\nu^* -
      \underline{\nu})
      + \sum_{i \in \negSet{\nu^*}} \costFun{i}{\underline{\nu}} x_i
      + \sum_{i \in \negSet{\nu^*}} (\nu^* - \underline{\nu}) x_i
    \\
      & + \sum_{i \in \posSet{\nu^*} \cap \leSet{\nu^*}}
        \costFun{i}{\underline{\nu}} x_i u^*_i
        + \sum_{i \in \posSet{\nu^*} \cap \leSet{\nu^*}} (\nu^* -
        \underline{\nu}) x_i u^*_i
    \\
      & + \sum_{i \in \posSet{\nu^*} \cap \geqSet{\nu^*}}
        \costFun{i}{\underline{\nu}} x_i u^*_i
    \\
    \overset{\phantom{\text{Lemma~\ref{lem:proj-ref}}}}{=}
    & \ \sum_{i=1}^n \nomCost_i x_i - k\underline{\nu}
      + \sum_{i \in \negSet{\nu^*}} \costFun{i}{\underline{\nu}} x_i
      + \sum_{i \in \posSet{\nu^*}} \costFun{i}{\underline{\nu}} x_i u^*_i
    \\
      & + (\nu^* - \underline{\nu})
        \left( \sum_{i \in \negSet{\nu^*}} x_i
        + \sum_{i \in \posSet{\nu^*} \cap \leSet{\nu^*}} x_i u^*_i - k \right)
    \\
    \overset{\phantom{\text{Lemma~\ref{lem:proj-ref}}}}{\leq}
    & \ \sum_{i=1}^n \nomCost_i x_i - k\underline{\nu}
      + \sum_{i \in \negSet{\nu^*}} \costFun{i}{\underline{\nu}} x_i
      + \sum_{i \in \posSet{\nu^*}} \costFun{i}{\underline{\nu}} x_i
      u^*_i
    \\
    \overset{\substack{\text{Lemma~\ref{lem:pos-neg-set}}\\\text{(i)}}}{=}
    & \ \sum_{i=1}^n \nomCost_i x_i - k\underline{\nu}
      + \sum_{i \in \negSet{\underline{\nu}}} \costFun{i}{\underline{\nu}}
      x_i
      + \sum_{i \in \posSet{\underline{\nu}}} \costFun{i}{\underline{\nu}}
      x_i u^*_i
    \\
      & \ + \sum_{i \in \posSet{\nu^*} \cap \negSet{\underline{\nu}}}
        \costFun{i}{\underline{\nu}} x_i \left( u^*_i - 1 \right)
    \\
    \overset{\substack{\text{Lemma~\ref{lem:pos-neg-set}}\\\text{(v)}}}{=}
    & \ \sum_{i=1}^n \nomCost_i x_i - k\underline{\nu}
      + \sum_{i \in \negSet{\underline{\nu}}} \costFun{i}{\underline{\nu}}
      x_i
      + \sum_{i \in \posSet{\underline{\nu}}} \costFun{i}{\underline{\nu}}
      x_i u^*_i,
  \end{align*}
  \ie, $(u^*,\underline{\nu})$ is an optimal
  solution to Problem~\eqref{eq:proj-ref-alpha} as well.
  
  Second, suppose that
  \begin{equation*}
    \sum_{i \in \negSet{\nu^*}} x_i  + \sum_{i \in \posSet{\nu^*} \cap
      \leSet{\nu^*}} x_i u^*_i > k
  \end{equation*}
  holds.
  Then, again by Lemmas~\ref{lem:proj-ref} and~\ref{lem:pos-neg-set}, we obtain
  \begin{align*}
    \beta(x)
    \overset{\text{Lemma~\ref{lem:proj-ref}}}{=}
    & \ \sum_{i=1}^n \nomCost_i x_i - k\nu^*
      + \sum_{i \in \negSet{\nu^*}} \costFun{i}{\nu^*} x_i
      + \sum_{i \in \posSet{\nu^*}} \costFun{i}{\nu^*} x_i u^*_i
    \\
    \overset{\phantom{\text{Lemma~\ref{lem:proj-ref}}}}{=}
    & \ \sum_{i=1}^n \nomCost_i x_i - k\nu^*
      + \sum_{i \in \negSet{\nu^*}} \costFun{i}{\nu^*} x_i
    \\
      & \ + \sum_{i \in \posSet{\nu^*} \cap \leSet{v^*}} \costFun{i}{\nu^*} x_i
        u^*_i+ \sum_{i \in \posSet{\nu^*} \cap \geqSet{v^*}} \costFun{i}{\nu^*}
        x_i u^*_i
    \\
    \overset{\substack{\text{Lemma~\ref{lem:pos-neg-set}}\\\text{(ii),(iii)}}}{=}
    & \ \sum_{i=1}^n \nomCost_i x_i - k\bar{\nu}
      - k(\nu^* - \bar{\nu})
      + \sum_{i \in \negSet{\nu^*}} \costFun{i}{\bar{\nu}} x_i
      - \sum_{i \in \negSet{\nu^*}} (\bar{\nu} - \nu^*) x_i
    \\
      & \ + \sum_{i \in \posSet{\nu^*} \cap \leSet{v^*}} \costFun{i}{\bar{\nu}} x_i
        u^*_i
        - \sum_{i \in \posSet{\nu^*} \cap \leSet{v^*}} (\bar{\nu} - \nu^*) x_i u^*_i
    \\
      & \ + \sum_{i \in \posSet{\nu^*} \cap \geqSet{v^*}} \costFun{i}{\bar{\nu}} x_i
        u^*_i
    \\
    \overset{\phantom{\text{Lemma~\ref{lem:proj-ref}}}}{=}
    & \ \sum_{i=1}^n \nomCost_i x_i - k\bar{\nu}
      + \sum_{i \in \negSet{\nu^*}} \costFun{i}{\bar{\nu}} x_i
      + \sum_{i \in \posSet{\nu^*}} \costFun{i}{\bar{\nu}} x_i u^*_i
    \\
      & \ + (\bar{\nu} - \nu^*)
        \left( k - \sum_{i \in \negSet{\nu^*}} x_i
        - \sum_{i \in \posSet{\nu^*} \cap \leSet{v^*}} x_i u^*_i \right)
    \\
    \overset{\phantom{\text{Lemma~\ref{lem:proj-ref}}}}{<}
    & \ \sum_{i=1}^n \nomCost_i x_i - k\bar{\nu}
      + \sum_{i \in \negSet{\nu^*}} \costFun{i}{\bar{\nu}} x_i
      + \sum_{i \in \posSet{\nu^*}} \costFun{i}{\bar{\nu}} x_i u^*_i
    \\
    \overset{\substack{\text{Lemma~\ref{lem:pos-neg-set}}\\\text{(i)}}}{=}
    & \ \sum_{i=1}^n \nomCost_i x_i - k\bar{\nu}
      + \sum_{i \in \negSet{\bar{\nu}}} \costFun{i}{\bar{\nu}} x_i
      + \sum_{i \in \posSet{\bar{\nu}}} \costFun{i}{\bar{\nu}} x_i u^*_i
    \\
      & \ + \sum_{i \in \negSet{\nu^*} \cap \posSet{\bar{\nu}}}
        \costFun{i}{\bar{\nu}} x_i \left( 1 - u^*_i \right)
    \\
    \overset{\substack{\text{Lemma~\ref{lem:pos-neg-set}}\\\text{(iv)}}}{=}
    & \ \sum_{i=1}^n \nomCost_i x_i - k\bar{\nu}
      + \sum_{i \in \negSet{\bar{\nu}}} \costFun{i}{\bar{\nu}} x_i
      + \sum_{i \in \posSet{\bar{\nu}}} \costFun{i}{\bar{\nu}} x_i u^*_i.
  \end{align*}
  The latter contradicts the optimality of~$(u^*,\nu^*)$ for
  Problem~\eqref{eq:proj-ref-alpha}.
  To sum up, we can thus assume, \Wlog, that $\nu^* \in \newUncertaintySet$
  holds. By Lemma~\ref{lem:proj-ref}, we thus obtain
  \begin{align*}
    \beta(x)
    & = \max_{\nu \in \newUncertaintySet}
      \Set{\max_{u \in \uncertaintySet_\Gamma} \Set{\sum_{i=1}^n \nomCost_i x_i
        + \sum_{i \in \negSet{\nu}} \costFun{i}{\nu} x_i
      + \sum_{i \in \posSet{\nu}} \costFun{i}{\nu} x_i u_i} - k\nu}
    \\
    & = \max_{\nu \in \newUncertaintySet}
      \Set{\max_{u \in \uncertaintySet_\Gamma} \Set{\sum_{i=1}^n
      \nomCostFun{i}{\nu} x_i
      + \sum_{i=1}^n \devCostFun{i}{\nu} x_i u_i} - k\nu},
  \end{align*}
  where the last equality is due to the definition of~$\nomCostFun{i}{\nu}$
  and~$\devCostFun{i}{\nu}$, $i \in N$, $\nu \in \newUncertaintySet$.
\end{proof}

Because of~$\Abs{\newUncertaintySet} \in O(n)$,
Proposition~\ref{prop:proj-ref} implies that
Problem~\eqref{eq:recourse-problem} can be solved in polynomial time
if~$\nomCost, \devCost \in \Q^n$.
Next, we state the projection-based formulation of
Problem~\eqref{eq:recoverable-robust-problem}.

\begin{theorem}
  \label{thm:proj-ref}
  Problem~\eqref{eq:recoverable-robust-problem} can be solved as the
  mixed-integer linear problem
  \begin{equation}
    \label{eq:proj-ref}
    \begin{split}
      \min_{x,\eta} \quad & \initCost^\top x + \eta
      \\
      \st \quad & \eta \geq \sum_{i=1}^n \scenarioCostFun{i}{\nu}{s} x_i -
                  k\nu,
                  \quad \nu \in \newUncertaintySet,\, s \in \scenarios(\nu),
      \\
                          & x \in X,\, \eta \in \R_{\geq 0},
    \end{split}
  \end{equation}
  where, for all~$\nu \in \newUncertaintySet$, we have
  \begin{equation*}
    \scenarios(\nu) \define \Defset{s \in \Set{1,2,\ldots, M}}{
      \scenarioCostFun{i}{\nu}{s} \define \nomCostFun{i}{\nu} +
      \devCostFun{i}{\nu} u^s_i,\, i \in N,\, u^s \in \uncertaintySet_\Gamma}.
  \end{equation*}
\end{theorem}
\begin{proof}
  By Proposition~\ref{prop:proj-ref},
  Problem~\eqref{eq:recoverable-robust-problem} can be solved as
  \begin{align*}
    \min_{x,\eta} \quad & \initCost^\top x + \eta
    \\
    \st \quad & \eta \geq \alpha(x,\nu) - k\nu,
                \quad \nu \in \newUncertaintySet,
    \\
                        & x \in X,\, \eta \in \R_{\geq 0},
  \end{align*}
  where, for given~$x \in X$ and~$\nu \in \newUncertaintySet$, we have
  \begin{align*}
    \alpha(x,\nu)
    & = \max_u \Defset{\sum_{i=1}^n \nomCostFun{i}{\nu} x_i
      + \sum_{i=1}^n \devCostFun{i}{\nu} x_i u_i}{
      u \in \uncertaintySet_\Gamma = \Set{u^1, u^2, \ldots, u^M}}.
  \end{align*}
  The claim then follows from the definition of the projected scenario
  set~$\scenarios(\nu)$.
\end{proof}

Note that, for given~$\nu \in \newUncertaintySet$, the set~$\scenarios(\nu)$ is
again a budgeted uncertainty set. Compared to the set~$\scenarios$, the only
difference lies in the definition of the nominal values and deviations, which
are adjusted according to~$\nu$.
We further note that no additional variables are introduced in the
projection-based formulation~\eqref{eq:proj-ref} to model uncertainties or
recovery actions. However, we still need to add a considerable number of
constraints to capture these aspects, which grows as~$O(n^{\Gamma + 1})$.

\subsection{Extended Formulation}
\label{sec:ext-form}

Our third reformulation of Problem~\eqref{eq:recoverable-robust-problem} builds
on the projection-based formulation in Section~\ref{sec:proj-ref}. In contrast
to the previous formulation, however, we now allow for an extended variable
space that includes additional continuous variables for modeling recovery
actions.

\begin{theorem}
  \label{thm:extended-formulation}
  Problem~\eqref{eq:recoverable-robust-problem} can be solved as the
  mixed-integer linear problem
  \begin{equation}
    \label{eq:extended-formulation}
    \begin{split}
      \min_{x,\eta,w,z} \quad & \initCost^\top x + \eta
      \\
      \st \quad \ & \eta \geq \sum_{i=1}^n \nomCostFun{i}{\nu} x_i
                    + \Gamma w^\nu +
                    \sum_{i=1}^n z^\nu_i - k\nu,
                    \quad \nu \in \newUncertaintySet,
      \\
                              & w^\nu + z^\nu_i \geq \devCostFun{i}{\nu}
                                x_i,
                                \quad i \in N,\,
                                \nu \in \newUncertaintySet,
      \\
                              & x \in X,\, \eta \in \R_{\geq 0},\, w^\nu
                                \in \R_{\geq 0},\, z^\nu \in \R^n_{\geq
                                0},
                                \quad \nu \in \newUncertaintySet.
    \end{split}
  \end{equation}
\end{theorem}

The proof of Theorem~\ref{thm:extended-formulation} can be found in
Appendix~\ref{sec:proof-thm-ext-form}.
In Problem~\eqref{eq:extended-formulation}, we consider an extended variable
space that includes additional continuous variables to model the uncertainties
and recovery actions, whose number is in~$O(n^2)$.
Overall, this results in a considerably smaller model compared to the scenario-
and projection-based formulations presented earlier.

\subsection{Compact Formulation}
\label{sec:compact-form}

We now state our fourth and last reformulation of
Problem~\eqref{eq:recoverable-robust-problem}, which is our most compact one.
For the derivation, we build on the extended
formulation~\eqref{eq:extended-formulation} and eliminate the continuous
variables~$w^\nu$ and~$z^\nu$, $\nu \in \newUncertaintySet$.
For notational convenience, let us set
\begin{equation*}
  \devCost_{n+1} \define 0
  \quad \text{and} \quad
  \nomCost_{n+1} \define 0
\end{equation*}
so that we obtain
\begin{equation*}
  \costFun{n+1}{\nu} = \min \Set{\devCost_{n+1},\, \nu - \nomCost_{n+1}} = 0,
  \quad \nu \in \newUncertaintySet.
\end{equation*}
For all~$\nu \in \newUncertaintySet$, let now~$\sigma_\nu : N \to N$ be a
perturbation function that orders the components of~$\Delta c(\nu)$ in
non-increasing order, \ie,
\begin{equation*}
  \devCostFun{\sigma_\nu(1)}{\nu} \geq \devCostFun{\sigma_\nu(2)}{\nu}
  \geq \cdots \geq \devCostFun{\sigma_\nu(n)}{\nu}.
\end{equation*}
Moreover, let~$\gamma$ be the largest odd integer such that~$\Gamma +
\gamma < n + 1$ holds and
\begin{equation*}
  \mathcal{L} \define \Set{\Gamma + 1, \Gamma + 3, \Gamma + 5, \dots,
    \Gamma + \gamma, n + 1}.
\end{equation*}
For all~$\nu \in \newUncertaintySet$ and~$\ell \in \mathcal{L}$, we further
define the modified cost functions
\begin{equation*}
  \compCostFun{\sigma_\nu(i)}{\nu}{\ell} \define
  \begin{cases}
    \nomCostFun{\sigma_\nu(i)}{\nu}
    + \devCostFun{\sigma_\nu(i)}{\nu} - \devCostFun{\sigma_\nu(\ell)}{\nu},
    & 1 \leq i \leq \ell,
    \\
    \nomCostFun{\sigma_\nu(i)}{\nu}, & \ell + 1 \leq i \leq n,
  \end{cases}
\end{equation*}
as well as protection terms
\begin{equation*}
  \kappa_\ell(\nu) \define \Gamma \devCostFun{\sigma_\nu(\ell)}{\nu} - k\nu.
\end{equation*}
Using the above notation, we obtain the following intermediate result.

\begin{proposition}
  \label{prop:compact-formulation}
  Problem~\eqref{eq:recoverable-robust-problem} is equivalent to
  \begin{subequations}
    \label{eq:compact-ref-aux}
    \begin{align}
      \min_{x,\eta} \quad & \initCost^\top x + \eta
      \\
      \st \quad & \eta
                  \geq \min_{\ell \in \mathcal{L}} \Set{\kappa_\ell(\nu)
                  + \sum_{i=1}^n \compCostFun{\sigma_\nu(i)}{\nu}{\ell}
                  x_{\sigma_\nu(i)}},
                  \quad \nu \in \newUncertaintySet,
                  \label{eq:compact-form-constr}
      \\
                          & x \in X,\, \eta \in \R_{\geq 0}.
    \end{align}
  \end{subequations}
\end{proposition}

The proof of Proposition~\ref{prop:compact-formulation} can be found in
Appendix~\ref{sec:proof-bs}. Because of~\eqref{eq:compact-form-constr},
Problem~\eqref{eq:compact-ref-aux} cannot be tackled by a general-purpose
solver directly. However, an equivalent MILP reformulation of
Problem~\eqref{eq:compact-ref-aux} can be obtained by modeling the selection
of~$\ell \in \mathcal{L}$ using auxiliary binary variables and SOS$1$
constraints. To this end, we further need to introduce the big-$M$ constants
\begin{equation*}
  M(\nu) \define \max_{\ell \in \mathcal{L}} \Set{
    \Gamma \devCostFun{\sigma_\nu(k)}{\nu}
    + \sum_{i=1}^n \max \Set{0,\, \compCostFun{\sigma_\nu(i)}{\nu}{\ell}}}
\end{equation*}
for all~$\nu \in \newUncertaintySet$.
Next, we state the compact MILP reformulation of
Problem~\eqref{eq:recoverable-robust-problem}.

\begin{theorem}
  \label{thm:compact-formulation}
  Problem~\eqref{eq:recoverable-robust-problem} can be solved as the
  mixed-integer linear problem
  \begin{subequations}
    \label{eq:compact-ref}
    \begin{align}
      \min_{x,\eta,y} \quad & \initCost^\top x + \eta
      \\
      \st \quad & \eta
                  \geq \kappa_\ell(\nu)
                  + \sum_{i=1}^n \compCostFun{\sigma_\nu(i)}{\nu}{\ell}
                  x_{\sigma_\nu(i)} - M(\nu)(1 - y^\nu_\ell),
                  \quad \ell \in \mathcal{L},\, \nu \in \newUncertaintySet,
      \\
                            & \sum_{\ell \in \mathcal{L}} y^\nu_\ell
                              = 1,
                              \quad \nu \in \newUncertaintySet,
                              \label{eq:compact-ref:sos1}
      \\
                            & x \in X,\, \eta \in \R_{\geq 0},\,
                              y^\nu \in \Set{0,1}^{\Abs{\mathcal{L}}},
                              \quad \nu \in \newUncertaintySet.
                              \label{eq:compact-ref:vars}
    \end{align}
  \end{subequations}
\end{theorem}
\begin{proof}
  We show that, for an optimal solution~$(x^*,\eta^*,y^*)$ to
  Problem~\eqref{eq:compact-ref}, the pair~$(x^*,\eta^*)$ is an optimal
  solution to Problem~\eqref{eq:compact-ref-aux}.
  We prove this by contradiction.
  To this end, let~$(x^*,\eta^*,y^*)$ be an optimal solution to
  Problem~\eqref{eq:compact-ref} and suppose that~$(x^*,\eta^*)$ does not
  solve Problem~\eqref{eq:compact-ref-aux}, \ie, there
  exists~$(\hat x,\hat \eta)$ feasible for Problem~\eqref{eq:compact-ref-aux}
  with~$\initCost^\top \hat x + \hat \eta < \initCost^\top x^* + \eta^*$.
  For all~$\nu \in \newUncertaintySet$, we define
  \begin{equation*}
    \ell^*_\nu \define \argmin_{\ell \in \mathcal{L}}
    \Set{\kappa_\ell(\nu) + \sum_{i=1}^n \compCostFun{\sigma_\nu(i)}{\nu}{\ell}
      \hat x_{\sigma_\nu(i)}}
    \quad \text{and} \quad
    \hat y^\nu_\ell \define
    \begin{cases}
      1, & \ell = \ell^*_\nu,
      \\
      0, & \ell \in \mathcal{L} \setminus \Set{\ell^*_\nu}.
    \end{cases}
  \end{equation*}
  By construction, $(\hat x,\hat \eta, \hat y)$
  with~$\hat y = (\hat y^\nu)_{\nu \in \newUncertaintySet}$
  satisfies~\eqref{eq:compact-ref:sos1} and~\eqref{eq:compact-ref:vars}.
  For all~$\nu \in \newUncertaintySet$, we further have
  \begin{align*}
    \hat \eta
    & \geq \min_{\ell \in \mathcal{L}} \Set{\kappa_\ell(\nu)
      + \sum_{i=1}^n \compCostFun{\sigma_\nu(i)}{\nu}{\ell} \hat
      x_{\sigma_\nu(i)}}
      = \kappa_{\ell^*_\nu}(\nu)
      + \sum_{i=1}^n \compCostFun{\sigma_\nu(i)}{\nu}{\ell^*_\nu} \hat
      x_{\sigma_\nu(i)}
    \\
    & = \kappa_{\ell^*_\nu}(\nu)
      + \sum_{i=1}^n \compCostFun{\sigma_\nu(i)}{\nu}{\ell^*_\nu} \hat
      x_{\sigma_\nu(i)} - M(\nu)(1 - y^\nu_{\ell^*_\nu})
  \end{align*}
  and
  \begin{align*}
    \hat \eta
    \geq & \ - k\nu + \Gamma \devCostFun{\sigma_\nu(\ell)}{\nu}
           + \sum_{i=1}^n \max \Set{0,\, \compCostFun{\sigma_\nu(i)}{\nu}{\ell}}
    \\
         & \ - \max_{\ell' \in \mathcal{L}} \Set{
           \Gamma \devCostFun{\sigma_\nu(k)}{\nu}
           + \sum_{i=1}^n \max \Set{0,\, \compCostFun{\sigma_\nu(i)}{\nu}{\ell'}}}
    \\
    = & \ \kappa_\ell(\nu)
        + \sum_{i=1}^n \max \Set{0,\, \compCostFun{\sigma_\nu(i)}{\nu}{\ell}}
        - M(\nu)
    \\
    \geq & \ \kappa_\ell(\nu)
           + \sum_{i=1}^n \max \Set{0,\,
           \compCostFun{\sigma_\nu(i)}{\nu}{\ell}} \hat x_{\sigma_\nu(i)}
           - M(\nu)
    \\
    \geq & \ \kappa_\ell(\nu)
           + \sum_{i=1}^n \max \Set{0,\,
           \compCostFun{\sigma_\nu(i)}{\nu}{\ell}} \hat x_{\sigma_\nu(i)}
           + \sum_{i=1}^n \min \Set{0,\,
           \compCostFun{\sigma_\nu(i)}{\nu}{\ell}} \hat x_{\sigma_\nu(i)}
           - M(\nu)
    \\
    = & \ \kappa_\ell(\nu) + \sum_{i=1}^n
        \compCostFun{\sigma_\nu(i)}{\nu}{\ell} \hat x_{\sigma_\nu(i)}
        - M(\nu)(1 - \hat{y}^\nu_{\ell})
  \end{align*}
  for all~$\ell \in \mathcal{L} \setminus \Set{\ell^*_\nu}$.
  Hence, the point~$(\hat x,\hat \eta,\hat y)$ is feasible for
  Problem~\eqref{eq:compact-ref}
  and has a better objective function value than~$(x^*,\eta^*,y^*)$.
  This is a contradiction to the optimality of~$(x^*,\eta^*,y^*)$.
  By Proposition~\ref{prop:compact-formulation}, this concludes the proof.
\end{proof}

The number of additional variables and constraints introduced for modeling the
uncertainties and recovery actions in Problem~\eqref{eq:compact-ref} also grows
as~$O(n^2)$. Compared to the extended formulation, however, the compact
formulation~\eqref{eq:compact-ref} does not make use of additional continuous
variables so that the combinatorial structure of the original problem is
preserved.

To conclude this section, we summarize the number of additional variables and
constraints introduced for modeling uncertainties and recovery actions
in the four presented reformulations of the~$k$-delete recoverable robust
problem~\eqref{eq:recoverable-robust-problem} in Table~\ref{tab:comp-refs}.

\begin{table}[b]
  \centering
  \caption{The number of additional variables (``\# variables'') and
    constraints (``\# constraints'') introduced for modeling uncertainties
    and recovery actions in the scenario-based, projection-based, extended, and
    compact reformulations of Problem~\eqref{eq:recoverable-robust-problem}. In
    particular, the number of additional continuous (``cont.'') and binary
    variables (``bin.'') is shown. Note that~$M \in O(n^\Gamma)$
    and~\mbox{$\Abs{\newUncertaintySet}, \Abs{\mathcal{L}} \in O(n)$}.}
  \begin{tabular}{lrrr}
    \toprule
    formulation & \multicolumn{2}{c}{\# variables} & \# constraints \\
    \cmidrule(lr){2-3}
    & cont. & bin. & \\
    \midrule
    scenario-based; see~\eqref{eq:scenario-ref}
                & $1$ & $Mn$ & $M(2 + n) + 1$\\
    projection-based; see~\eqref{eq:proj-ref}
                & $1$ & $0$ & $M\Abs{\newUncertaintySet} + 1$ \\
    extended; see~\eqref{eq:extended-formulation}
                & $\Abs{\newUncertaintySet}(n + 1) + 1$ & $0$
    & $2\Abs{\newUncertaintySet}(n + 1) + 1$ \\
    compact; see~\eqref{eq:compact-ref}
                & $1$ & $\Abs{\newUncertaintySet}\Abs{\mathcal{L}}$
    & $\Abs{\newUncertaintySet}(\Abs{\mathcal{L}} + 1) + 1$ \\
    \bottomrule
  \end{tabular}
  \label{tab:comp-refs}
\end{table}


\section{Exact Solution Approaches}
\label{sec:solution-approaches}

We now present exact solution approaches for the~$k$-delete
recoverable robust problem~\eqref{eq:recoverable-robust-problem}, which are
based on the four reformulations derived in
Section~\ref{sec:reformulations}. We consider three approaches:
\begin{enumerate}
\item solving the full model using a general-purpose MILP solver
  (MILP approach),
\item column-and-constraint generation (CCG), and
\item branch-and-cut (BnC).
\end{enumerate}
Table~\ref{tab:approaches} provides an overview of which methods are applicable
for each formulation.
Overall, we present eight approaches to solve the~$k$-delete recoverable robust
problem~\eqref{eq:recoverable-robust-problem}.
In Section~\ref{sec:milp}, we elaborate on the MILP
approach. Afterward, in Section~\ref{sec:ccg}, we present the
column-and-constraint generation algorithms.
Finally, we discuss the branch-and-cut frameworks in Section~\ref{sec:bnc}.
\begin{table}
  \centering
  \caption{Solution methods applicable to each formulation: solving the full
    model using a general-purpose MILP solver (MILP), column-and-constraint
    generation (CCG), and branch-and-cut (BnC).}
  \begin{tabular}{lccc}
    \toprule
    formulation & MILP & CCG & BnC \\
    \midrule
    scenario-based; see~\eqref{eq:scenario-ref}
                & \cross & \check & \cross \\
    projection-based; see~\eqref{eq:proj-ref}
                & \cross & \cross & \check \\
    extended; see~\eqref{eq:extended-formulation}
                & \check & \check & \check \\
    compact; see~\eqref{eq:compact-ref}
                & \check & \check & \check \\
    \bottomrule
  \end{tabular}
  \label{tab:approaches}
\end{table}

\subsection{MILP Approach}
\label{sec:milp}

For the extended formulation~\eqref{eq:extended-formulation} and the compact
formulation~\eqref{eq:compact-ref}, the number of additional
variables and constraints introduced for modeling uncertainties and recovery
actions is in~$O(n^2)$; cf.\ Table~\ref{tab:comp-refs}. Hence, these two
reformulations can, in principle, be solved directly using a general-purpose
MILP solver.
In practice, however, the size of the resulting MILP can still become
restrictive for larger instances and the LP bound may be rather weak.
We therefore also present CCG and BnC frameworks for these formulations.
In contrast, general-purpose MILP solvers cannot tackle the scenario-based
formulation~\eqref{eq:scenario-ref} and the projection-based
formulation~\eqref{eq:proj-ref} directly because of their significant number of
additional variables and constraints.

\subsection{Column-and-Constraint Generation}
\label{sec:ccg}

Column-and-constraint generation \parencite{Zeng_Zhao:2013} is an iterative
framework that can be 
used to solve robust optimization problems with a large number of
scenarios. The idea is to start from a relaxation of
the problem in which only a subset of the variables and constraints
associated with specific scenarios is included.
New scenarios, along with their corresponding variables and constraints, are
then generated and added iteratively.
To achieve this, the method alternates between solving a
\emph{(relaxed) master problem} and an \emph{adversarial problem}, where the
latter is used to identify critical scenarios that are missing from the current
master problem.
We now discuss how CCG is applied to the scenario-based, compact, and extended
reformulations of Problem~\eqref{eq:recoverable-robust-problem}.

\subsubsection{Scenario-Based Formulation}

Problem~\eqref{eq:scenario-ref} involves~$O(n^\Gamma)$ scenarios, which may
render enumerating all of them impractical. Applying CCG, we thus start from a
relaxation of the problem in which we only consider a subset of scenarios.
In iteration~$j$ of the method, we consider the (relaxed) master problem
\begin{equation}
  \label{eq:subset-scenario-ref}
  \begin{split}
    \min_{x,y,\eta} \quad & \initCost^\top x + \eta
    \\
    \st \quad & \eta \geq \sum_{i=1}^n \scenarioCost{s}_i y^s_i,
                \quad s \in \scenarios^j,
    \\
                          & \sum_{i=1}^n y^s_i \geq \sum_{i=1}^n x_i - k,
                            \quad s \in \scenarios^j,
    \\
                          & x \geq y^s, \quad s \in \scenarios^j,
    \\
                          & x \in X,\, \eta \in \R_{\geq 0},\, y^s \in
                            \Set{0,1}^n, \quad s \in \scenarios^j,
  \end{split}
\end{equation}
with~$\scenarios^j \subseteq \scenarios$.
To identify critical scenarios that are missing from this problem, we then
determine the worst-case recovery costs for a given~$x \in X$ by solving the
adversarial problem
\begin{equation}
  \label{eq:adversarial-prob}
  \max_{s \in \scenarios} \Set{R(x,s)}
  = \max_{s \in \scenarios} \Set{\min_{y \in \Set{0,1}^n} \Defset{\sum_{i=1}^n
    c^s_iy_i}{y \leq x,\, \sum_{i=1}^n y_i \geq \sum_{i=1}^n x_i - k}}.
\end{equation}
In Problem~\eqref{eq:adversarial-prob}, we consider the entire
scenario set~$\scenarios$ again. To tackle this problem effectively, we
thus resort to solving an equivalent compact reformulation, which we state in
the following proposition.

\begin{proposition}
  \label{prop:adversarial-prob-ref}
  Let~$x \in X$ be given arbitrarily. Then, Problem~\eqref{eq:adversarial-prob}
  can be solved as the mixed-integer linear problem
  \begin{equation}
    \label{eq:adversarial-prob-ref}
    \begin{split}
      \max_{z,\lambda,\mu} \quad
      & \left( \sum_{i=1}^n x_i - k\right) \lambda - \sum_{i=1}^n x_i\mu_i
      \\
      \st \quad \, & \sum_{i=1}^n z_i \leq \Gamma,
      \\
      & \lambda - \mu_i \leq \nomCost_i + \devCost_i z_i,
        \quad i \in N,
      \\
      & z \in \Set{0,1}^n,\, \mu \in \R^n_{\geq 0},\, \lambda \in
        \R_{\geq 0}.
    \end{split}
  \end{equation}
\end{proposition}

The proof of Proposition~\ref{prop:adversarial-prob-ref} can be found in
Appendix~\ref{sec:proof-adversarial-prob}.
The CCG algorithm for solving the~$k$-delete recoverable robust
problem~\eqref{eq:recoverable-robust-problem} using the scenario-based
formulation is formally stated in Algorithm~\ref{alg:ccg}.
In Line~\ref{alg:ccg:init-bounds}, we initialize lower and upper bounds
for Problem~\eqref{eq:recoverable-robust-problem} with~$L$ and~$U$,
respectively. As long as the optimality gap is not closed, i.e., $U > L$ holds,
we alternate between solving the master problem
(Line~\ref{alg:ccg:solve-master}) and the MILP reformulation of the adversarial
problem (Line~\ref{alg:ccg:solve-adversarial-prob}).
In Line~\ref{alg:ccg:add-scenario}, the~$z^j$ component of an optimal solution
to the latter is then used to generate a new scenario that is
missing from the current formulation.
Note that we need to specify an initial subset of scenarios~$\scenarios^0$ in
Algorithm~\ref{alg:ccg}, which can be done in various ways.
A straightforward initialization is to consider only the nominal scenario, i.e,
to set~$\scenarios^0 \define \defset{s=1}{\scenarioCost{s} = \nomCost}$.

\begin{algorithm}
  \begin{algorithmic}[1]
    \REQUIRE An instance of Problem~\eqref{eq:recoverable-robust-problem}, a
    subset of scenarios~$\scenarios^0 \subseteq \scenarios$,
    exact solution methods for Problems~\eqref{eq:subset-scenario-ref}
    and~\eqref{eq:adversarial-prob-ref}
    \ENSURE An optimal solution to
    Problem~\eqref{eq:recoverable-robust-problem}
    \STATE Set~$L \gets - \infty$ and~$U \gets + \infty$.
    \label{alg:ccg:init-bounds}
    \FOR{$j = 0,1,\ldots$}
    \STATE Compute an optimal solution~$(x^j,y^j,\eta^j)$ to
    Problem~\eqref{eq:subset-scenario-ref} with the
    set~$\scenarios^j$.
    \label{alg:ccg:solve-master}
    \STATE Set~$L \gets \initCost^\top x^j + \eta^j$.
    \label{alg:ccg:update-lb}
    \IF{$L \geq U$}
    \RETURN $x^j$
    \ENDIF
    \STATE Compute an optimal solution~$(z^j,\lambda^j,\mu^j)$
    to the~$x^j$-parameterized problem~\eqref{eq:adversarial-prob-ref}
    and let~$R(x^j)$ denote its optimal objective function value.
    \label{alg:ccg:solve-adversarial-prob}
    \STATE Set~$U \gets \min \set{U,\, \initCost^\top x^j + R(x^j)}$.
    \label{alg:ccg:update-ub}
    \IF{$L \geq U$}
    \RETURN $x^j$
    \ENDIF
    \STATE Define a new scenario~$s$ with~$\scenarioCost{s}_i \gets
    \nomCost_i + \devCost_i z^j_i$ for all~$i \in N$.
    \STATE Set~$\scenarios^{j+1} \gets \scenarios^j \cup \set{s}$
    and~$j \gets j+1$.
    \label{alg:ccg:add-scenario}
    \ENDFOR
  \end{algorithmic}
  \caption{CCG for the Scenario-Based Formulation}
  \label{alg:ccg}
\end{algorithm}

\begin{theorem}
  \label{thm:ccg-correct}
  Algorithm~\ref{alg:ccg} terminates after finitely many iterations with a
  globally optimal solution to the~$k$-delete recoverable robust
  problem~\eqref{eq:recoverable-robust-problem}.
\end{theorem}
\begin{proof}
  Let~$(x^j,y^j,\eta^j)$ be an optimal solution to the
  problem solved in Line~\ref{alg:ccg:solve-master} of
  Algorithm~\ref{alg:ccg} for some~$j \in \N$.
  Because~\eqref{eq:subset-scenario-ref} is a relaxation of
  Problem~\eqref{eq:scenario-ref}, $\initCost^\top x^j + \eta^j$
  is a valid lower bound for its optimal objective function value.
  Moreover, due to Proposition~\ref{prop:adversarial-prob-ref}, solving
  Problem~\eqref{eq:adversarial-prob-ref} for given~$x^j \in X$ yields a valid
  upper bound for Problem~\eqref{eq:scenario-ref}.
  Hence, the bound updates in Lines~\ref{alg:ccg:update-lb}
  and~\ref{alg:ccg:update-ub} are correct and~$L \leq U$ holds.
  If the termination criterion~$L \geq U$ is satisfied, we thus have~$L = U$,
  \ie, the optimality gap is closed.
  Finite termination now follows from the finiteness of the scenario
  set~$\scenarios$, the finiteness of the branch-and-cut methods used to solve
  the MILPs in Lines~\ref{alg:ccg:solve-master}
  and~\ref{alg:ccg:solve-adversarial-prob}, and from
  the fact that, in the worst case, the problem considered in
  Line~\ref{alg:ccg:solve-master} is the one in which all scenarios have been
  added, \ie, we solve Problem~\eqref{eq:scenario-ref}.
  The claim then follows from the equivalence of
  Problems~\eqref{eq:recoverable-robust-problem} and~\eqref{eq:scenario-ref}.
\end{proof}

Next, we present the CCG method applied to the compact formulation of
the~$k$-delete recoverable robust
problem~\eqref{eq:recoverable-robust-problem}.

\subsubsection{Compact Formulation}
\label{sec:ccg:compact-formulation}

In contrast to the scenario-based formulation~\eqref{eq:scenario-ref}, the
compact formulation~\eqref{eq:compact-ref} does not rely on the scenario
set~$\scenarios$ but on the sets~$\newUncertaintySet$ and~$\mathcal{L}$
instead. Although~$\Abs{\newUncertaintySet}\Abs{\mathcal{L}} \in O(n^2)$,
Problem~\eqref{eq:compact-ref} can still become large for practical instances.
Hence, applying CCG may also be beneficial in this setting.
In iteration~$j$ of the CCG method, we consider the (relaxed) master problem
\begin{equation}
  \label{eq:subset-compact-ref}
  \begin{split}
    \min_{x,\eta,y} \quad & \initCost^\top x + \eta
    \\
    \st \quad & \eta
                \geq \kappa_\ell(\nu)
                + \sum_{i=1}^n \compCostFun{\sigma_\nu(i)}{\nu}{\ell}
                x_{\sigma_\nu(i)}  - M(\nu)(1 - y^\nu_\ell),
                \quad \ell \in \mathcal{L},\, \nu \in \newUncertaintySet^j,
    \\
                          & \sum_{\ell \in \mathcal{L}} y^\nu_\ell
                            = 1,
                            \quad \nu \in \newUncertaintySet^j,
    \\
                          & x \in X,\, \eta \in \R_{\geq 0},\,
                            y^\nu \in \Set{0,1}^{\Abs{\mathcal{L}}},
                            \quad \nu \in \newUncertaintySet^j,
  \end{split}
\end{equation}
where~$\newUncertaintySet^j \subseteq \newUncertaintySet$.
Due to Proposition~\ref{prop:proj-ref}, the worst case for a given~$x \in X$
and~$\nu \in \newUncertaintySet$ can be obtained by solving
\begin{equation}
  \label{eq:adversarial-prob:compact}
  \max_{u} \quad \sum_{i=1}^n \left( \nomCostFun{i}{\nu} +
      \devCostFun{i}{\nu} u_i \right) x_i - k\nu
  \quad \st \quad u \in \uncertaintySet_\Gamma.
\end{equation}
Solving Problem~\eqref{eq:adversarial-prob:compact} avoids computing the cost
coefficients~$\compCostFun{i}{\nu}{\ell}$ for all~$i \in N$, 
$\ell \in \mathcal{L}$, and~$\nu \in \newUncertaintySet$, which may
be quite costly in practice. Moreover, we note that
Problem~\eqref{eq:adversarial-prob:compact} can be solved efficiently using a 
classic greedy algorithm.

We formally state the CCG method for the compact formulation in
Algorithm~\ref{alg:ccg:compact-ref}.
The first steps of the method are similar to those in Algorithm~\ref{alg:ccg}.
However, in Line~\ref{alg:ccg:compact-ref:solve-subprob}, we now solve an
auxiliary problem for fixed~$x^j$ across
all~\mbox{$\nu \in \newUncertaintySet$}.
Note that the problems considered in
Line~\ref{alg:ccg:compact-ref:solve-subprob} are independent of each other,
i.e., they can be solved in parallel.
Although one can, in principle, augment the set~$\newUncertaintySet^j$ with
every~$\nu \in \newUncertaintySet$ for which~\mbox{$\eta^j < R(x^j,\nu)$}
holds, we only add the worst-case uncertainty realization~$\bar \nu$ to control
the number of variables and constraints added in each iteration; see
Lines~\ref{alg:ccg:compact-ref:update-ub}--\ref{alg:ccg:compact-ref:add-scenario}.
Moreover, the for-loop in Line~\ref{alg:ccg:compact-ref:for-loop} can, in
principle, be terminated as soon as a~$\nu \in \newUncertaintySet$
with~$\eta^j < R(x^j,\nu)$ is found. This can help reduce the computational
burden of the CCG method.
In this case, however, the upper bound cannot be updated in
Line~\ref{alg:ccg:compact-ref:update-ub}
as~$R(x^j,\nu) < \max_{\bar \nu \in \newUncertaintySet} \set{R(x^j,\bar \nu)}$
may hold.
To achieve a trade-off between the frequency at which the upper bound is
updated and the computational burden of the method,
we solve Problem~\eqref{eq:adversarial-prob:compact} for all~$\nu \in
\newUncertaintySet$ only in selected iterations, e.g., every
$42$nd iteration. In what follows, we refer to this as the
\emph{full evaluation frequency}.
Finally, we note that finite termination of
Algorithm~\ref{alg:ccg:compact-ref} with a globally
optimal solution to Problem~\eqref{eq:recoverable-robust-problem} can be shown
in analogy to the proof of Theorem~\ref{thm:ccg-correct}.

\begin{algorithm}
  \begin{algorithmic}[1]
    \REQUIRE An instance of Problem~\eqref{eq:recoverable-robust-problem},
    exact solution methods for Problems~\eqref{eq:subset-compact-ref}
    and~\eqref{eq:adversarial-prob:compact}
    \ENSURE An optimal solution to
    Problem~\eqref{eq:recoverable-robust-problem}
    \STATE Set~$L \gets - \infty$, $U \gets + \infty$,
    and~$\newUncertaintySet^0 \gets \set{0}$.
    \FOR{$j = 0,1,\ldots$}
    \STATE Compute an optimal solution~$(x^j,\eta^j,y^j)$ to
    Problem~\eqref{eq:subset-compact-ref} with the
    set~$\mathcal{\newUncertaintySet}^j$.
    \label{alg:ccg:compact-ref:solve-master}
    \STATE Set~$L \gets \initCost^\top x^j + \eta^j$.
    \IF{$L \geq U$}
    \RETURN $x^j$
    \ENDIF
    \FOR{$\nu \in \newUncertaintySet$}    
    \label{alg:ccg:compact-ref:for-loop}
    \STATE Solve the~$(x^j,\nu)$-parameterized
    problem~\eqref{eq:adversarial-prob:compact} and let~$R(x^j,\nu)$ denote its
    optimal objective function value.
    \label{alg:ccg:compact-ref:solve-subprob}
    \ENDFOR
    \STATE Set~$\bar \nu \gets \argmax_{\nu \in \newUncertaintySet}
    \set{R(x^j,\nu)}$ and~$U \gets \min \set{U,\, \initCost^\top x^j +
      R(x^j,\bar \nu)}$.
    \label{alg:ccg:compact-ref:update-ub}
    \IF{$L \geq U$}
    \RETURN $x^j$
    \ENDIF
    \IF{$\eta^j < R(x^j,\bar \nu)$}
    \STATE Set~$\newUncertaintySet^{j+1} \gets \newUncertaintySet^j
    \cup \set{\bar \nu}$ and~$j \gets j+1$.
    \label{alg:ccg:compact-ref:add-scenario}
    \ENDIF
    \ENDFOR
  \end{algorithmic}
  \caption{CCG for the Compact Formulation}
  \label{alg:ccg:compact-ref}
\end{algorithm}

\subsubsection{Extended Formulation}
\label{sec:ccg:extended-formulation}

Similar to the compact formulation, the extended
formulation~\eqref{eq:extended-formulation} is based on the projected
uncertainty set~$\newUncertaintySet$ rather than the scenario set~$\scenarios$.
The CCG framework from Algorithm~\ref{alg:ccg:compact-ref} can also be applied
to the extended formulation, with the only difference being the (relaxed)
master problem considered in Line~\ref{alg:ccg:compact-ref:solve-master} of the
method. For iteration~$j$ with~$\newUncertaintySet^j \subseteq
\newUncertaintySet$, this problem is given by
\begin{equation*}
  \begin{split}
    \min_{x,\eta,w,z} \quad & \initCost^\top x + \eta
    \\
    \st \quad \ & \eta \geq \sum_{i=1}^n \nomCostFun{i}{\nu} x_i
                  + \Gamma w^\nu +
                  \sum_{i=1}^n z^\nu_i - k\nu,
                  \quad \nu \in \newUncertaintySet^j,
    \\
                            & w^\nu + z^\nu_i \geq \devCostFun{i}{\nu}
                              x_i,
                              \quad i \in N,\,
                              \nu \in \newUncertaintySet^j,
    \\
                            & x \in X,\, \eta \in \R_{\geq 0},\, w^\nu
                              \in \R_{\geq 0},\, z^\nu \in \R^n_{\geq
                              0},
                              \quad \nu \in \newUncertaintySet^j.
  \end{split}
\end{equation*}

\subsection{Branch-and-Cut}
\label{sec:bnc}

To solve the projection-based, compact, and extended formulations of
Problem~\eqref{eq:recoverable-robust-problem}, we now present branch-and-cut
frameworks that are similar to (generalized) Benders decomposition
\parencite{Benders:1962,Geoffrion:1972}.
In these methods, we only add constraints associated with violated scenarios
and keep the sets of variables fixed.
This is in contrast to the CCG methods discussed earlier, in which both
variables and constraints are added iteratively.
We first present the branch-and-cut framework for the projection-based
formulation and then discuss what needs to be adapted to tackle the extended
and the compact formulations.

\subsubsection{Projection-Based Formulation}

In the projection-based formulation~\eqref{eq:proj-ref},
we only add one additional continuous variable (the variable~$\eta$ used for
the epigraph reformulation) for modeling the uncertainties and the recovery
actions. Hence, we can tackle Problem~\eqref{eq:proj-ref} using branch-and-cut.
We start by solving the linear problem
\begin{equation}
  \label{eq:root-node-problem}
  \min_{x,\eta} \quad  \initCost^\top x + \eta
  \quad \st \quad (x,\eta) \in \Omega_0 \define \Set{(\bar x, \bar \eta) \in
    \bar X \times \R_{\geq 0}},
\end{equation}
where~$\bar{X} \subseteq \R^n$ is a continuous relaxation of~$X$,
\ie, the integer points contained in~$\bar{X}$ coincide with~$X$.
After solving Problem~\eqref{eq:root-node-problem}, we iteratively augment the
set~$\Omega_0$ to ensure integer feasibility and to approximate the worst-case
recovery costs until a solution~$(x^*,\eta^*)$ to
Problem~\eqref{eq:root-node-problem} is also feasible for
Problem~\eqref{eq:proj-ref}.
At node~$j$ of the branch-and-cut search tree, we consider the problem
\begin{equation}
  \label{eq:j-node-problem}
  \min_{x,\eta} \quad  \initCost^\top x + \eta
  \quad \st \quad (x,\eta) \in \Omega_j \subseteq \R^n \times \R_{\geq 0},
\end{equation}
where the set~$\Omega_j$ is obtained from~$\Omega_0$ by adding all valid
inequalities that have been generated along the path from the root to
node~$j$---both to include new scenarios for the uncertainty and to
separate fractional solutions---together with all branching decisions made
along that path.
We state the method for processing node~$j$ of the branch-and-cut search tree
in Algorithm~\ref{alg:node-processing:projection}.

\begin{algorithm}
  \begin{algorithmic}[1]
    \REQUIRE Exact solution methods for
    Problems~\eqref{eq:adversarial-prob:compact} and~\eqref{eq:j-node-problem},
    an upper bound~$U$
    \ENSURE An indication of whether node~$j$ is fathomed or whether two new
    sub-problems are generated due to branching
    \STATE Set~$\texttt{resolve} \gets \textsf{False}$.
    \label{alg:node-processing:projection:step-1}
    \STATE Solve Problem~\eqref{eq:j-node-problem}.
    \label{alg:node-processing:projection:solve-master}
    \IF{Problem~\eqref{eq:j-node-problem} is infeasible}
    \STATE Fathom the current node, i.e., go back to the main method.
    \label{alg:node-processing:projection:fathom-1}
    \ENDIF
    \STATE Let~$(x^j,\eta^j)$ denote the solution to
    Problem~\eqref{eq:j-node-problem}.
    \IF{$\initCost^\top x^j + \eta^j \geq U$}
    \STATE Fathom the current node, i.e., go back to the main method.
    \label{alg:node-processing:projection:fathom-2}
    \ENDIF
    \IF{$x^j \notin X$}
    \label{alg:node-processing:projection:if}
    \STATE Either generate cuts valid for $\Omega_j \cap (X
    \times \R_{\geq 0})$, augment~$\Omega_j$, and go to
    Step~\ref{alg:node-processing:projection:solve-master}, or branch and go
    back to the main method.
    \label{alg:node-processing:projection:integer-infeas}
    \ENDIF
    \FOR{$\nu \in \newUncertaintySet$}
    \STATE Compute an optimal solution~$u^j$ to
    the~$(x^j,\nu)$-parameterized problem~\eqref{eq:adversarial-prob:compact}
    and let~$R(x^j,\nu)$ denote its optimal objective function value.
    \label{alg:node-processing:projection:solve-subprob}
    \IF{$\eta^j < R(x^j,\nu)$}
    \STATE Set
    \begin{equation*}
      \Omega_j \gets \Omega_j \cap \Defset{(x,\eta)}{%
        \eta \geq \sum_{i=1}^n \left( \nomCostFun{i}{\nu} +
        \devCostFun{i}{\nu} u^j_i \right) x_i - k\nu}
    \end{equation*}
    and~$\texttt{resolve} \gets \textsf{True}$.
    \label{alg:node-processing:projection:add-cut}
    \ENDIF
    \ENDFOR
    \IF{\texttt{resolve}}
    \STATE Go to Step~\ref{alg:node-processing:projection:step-1}.
    \ENDIF
    \STATE Set~$U \gets \initCost^\top x^j + \eta^j$, update the incumbent
    solution, and fathom the current node, i.e., go back to the main method.
    \label{alg:node-processing:projection:fathom-3}
  \end{algorithmic}
  \caption{Processing Node~$j$ Using the Projection-Based Formulation}
  \label{alg:node-processing:projection}
\end{algorithm}

If Problem~\eqref{eq:j-node-problem} is infeasible or if its optimal objective
function value exceeds the current upper bound~$U$, we fathom node~$j$;
see Lines~\ref{alg:node-processing:projection:fathom-1}
and~\ref{alg:node-processing:projection:fathom-2} in
Algorithm~\ref{alg:node-processing:projection}.
Otherwise, we do the following.
First, we check if the variables~$x^j$ satisfy the integrality
constraints, \ie, we check if~$x^j \in X \subseteq \set{0,1}^n$ holds.
In the case of fractional solutions, we can use standard cutting
planes from mixed-integer linear optimization, see, \eg,
\textcite{Cornuejols:2008,Clautiaux_Ljubic:2024}, or branch to
separate~$(x^j,\eta^j)$; see
Line~\ref{alg:node-processing:projection:integer-infeas}.
If~$x^j \in X$ holds, we proceed by checking whether~$(x^j,\eta^j)$ satisfies
\begin{equation*}
  \eta^j \geq \sum_{i=1}^n \scenarioCostFun{i}{\nu}{s} x^j_i -
  k\nu,
  \quad \nu \in \newUncertaintySet,\, s \in \scenarios(\nu).
\end{equation*}
To this end, we solve Problem~\eqref{eq:adversarial-prob:compact} with~$x=x^j$
for all~$\nu \in \newUncertaintySet$.
Because these problems are independent of each other, we can solve them in
parallel. If
\begin{equation*}
  \eta^j
  \geq \max_{u} \Defset{\sum_{i=1}^n \left(
      \nomCostFun{i}{\nu} + \devCostFun{i}{\nu} u_i \right) x^j_i -
    k\nu}{u \in \uncertaintySet_\Gamma}
\end{equation*}
holds for all~$\nu \in \newUncertaintySet$, we update the incumbent and fathom
the current node; see Line~\ref{alg:node-processing:projection:fathom-3}.
Otherwise, we augment the set~$\Omega_j$ with a valid inequality that separates
the point~$(x^j,\eta^j)$; see
Line~\ref{alg:node-processing:projection:add-cut}.
Note that a cut is added for each~$\nu \in \newUncertaintySet$ for
which~$\eta^j < R(x^j,\nu)$ holds, i.e., up to~$\Abs{\newUncertaintySet}$ cuts
may be added at each node. However, it is also valid to consider, \eg, adding
only the most violated cut for the given~$x^j \in X$.
We present different cut separation strategies in
Section~\ref{sec:cut-strategies}.

\begin{theorem}
  \label{thm:bnc-correct}
  If we embed Algorithm~\ref{alg:node-processing:projection} into a
  classic branch-and-bound framework, we obtain a method that terminates
  with~$(x^*,\eta^*)$, where~$x^*$ is an optimal solution to the~$k$-delete
  recoverable robust problem~\eqref{eq:recoverable-robust-problem}, after
  finitely many iterations and after adding an overall finite number of cuts.
\end{theorem}
\begin{proof}
  Finite termination follows from the finiteness of the number of feasible
  first-stage decisions~$x \in X \subseteq \set{0,1}^n$,
  the finiteness of the branch-and-cut method used to solve the separation
  problem~\eqref{eq:adversarial-prob:compact},
  and from the fact that a pair~$(x,\eta)$ cannot occur twice during the
  execution of the method.
  We show the latter by contradiction. To this end, suppose that there exist
  solutions~$(x^{j},\eta^{j}) = (x^{l},\eta^{l})$ at nodes~$j$ and~$l$
  with~$j < l$. Then, by Line~\ref{alg:node-processing:projection:add-cut} of
  Algorithm~\ref{alg:node-processing:projection}, we have
  \begin{equation*}
    \eta^{j} = \eta^{l}
    \geq \sum_{i=1}^n \scenarioCostFun{i}{\nu}{s} x^{j}_i - k\nu,
    \quad \nu \in \newUncertaintySet,\, s \in \scenarios(\nu),
  \end{equation*}
  \ie, the termination criterion was already satisfied at node~$j$.
  Hence, an optimal solution cannot be overlooked.
  Finally, the number of cuts possibly added to the problem formulation is
  finite because~$\Abs{X},\Abs{\newUncertaintySet} < \infty$.
\end{proof}

Next, we discuss how to adapt the node processing procedure in
Algorithm~\ref{alg:node-processing:projection} so that it can be applied to the
compact formulation of Problem~\eqref{eq:recoverable-robust-problem}.

\subsubsection{Compact Formulation}

In the compact formulation~\eqref{eq:compact-ref}, the number of variables
needed to represent uncertainties and recovery actions is
in~$O(n^2)$; cf.\ Table~\ref{tab:comp-refs}. Hence, we can keep the full set of
variables in the master problem and generate the associated constraints on the
fly.
To process nodes using the compact reformulation~\eqref{eq:compact-ref}, we
thus make the following modifications to
Algorithm~\ref{alg:node-processing:projection}:
\begin{enumerate}
\item Instead of Problem~\eqref{eq:j-node-problem}, the problem considered at
  node~$j$ is
  \begin{equation*}
    \min_{x,\eta,y} \quad  \initCost^\top x + \eta
    \quad \st \quad (x,\eta,y) \in \Omega_j \subseteq \R^n \times \R_{\geq 0}
    \times [0,1]^{\Abs{\mathcal{L}} \times \Abs{\newUncertaintySet}},
  \end{equation*}
  i.e., we now operate in the~$(x,\eta,y)$-space. The set~$\Omega_j$ is
  obtained from
  \begin{equation*}
    \Omega_0 \define \Defset{(x,\eta,y) \in \bar X \times \R_{\geq 0} \times
      [0,1]^{\Abs{\mathcal{L}} \times \Abs{\newUncertaintySet}}}{%
      \sum_{\ell \in \mathcal{L}} y^\nu_\ell = 1,\, \nu \in \newUncertaintySet}
  \end{equation*}
  by adding all valid inequalities that have been generated along the path from
  the root to node~$j$ as well as all branching decisions made along that path.
\item In Line~\ref{alg:node-processing:projection:if} of
  Algorithm~\ref{alg:node-processing:projection}, we additionally check
  whether the~$y^j$ component of a solution~$(x^j,\eta^j,y^j)$ satisfies all
  integrality constraints. As before, fractional solutions are handled by
  either cutting or branching.
\item If a solution~$(x^j,\eta^j,y^j)$ to the problem at node~$j$
  satisfies~$\eta^j < R(x^j,\nu)$ for
  some~$\nu \in \newUncertaintySet$, we augment the
  set~$\Omega_j$ with~$\Abs{\mathcal{L}}$ inequality constraints, namely
  \begin{equation*}
    \eta \geq \kappa_\ell(\nu)
    + \sum_{i=1}^n \compCostFun{\sigma_\nu(i)}{\nu}{\ell} x_{\sigma_\nu(i)}
    - M(\nu)(1 - y^\nu_\ell),\quad \ell \in \mathcal{L},
  \end{equation*}
  in Line~\ref{alg:node-processing:projection:add-cut} of
  Algorithm~\ref{alg:node-processing:projection}.
\end{enumerate}
Finite termination of the overall branch-and-cut method in which this
modified node processing scheme is embedded can be shown in analogy to the
proof of Theorem~\ref{thm:bnc-correct}.

\enlargethispage{1cm}

\subsubsection{Extended Formulation}

The setting for the extended formulation~\eqref{eq:extended-formulation} is
very similar to that of the compact formulation, with the main difference
being the considered variable space. To process node~$j$ of the branch-and-cut
search tree using the extended formulation, we replace
Problem~\eqref{eq:j-node-problem} in
Algorithm~\ref{alg:node-processing:projection} with
\begin{equation*}
  \min_{x,\eta,w,z} \quad  \initCost^\top x + \eta
  \quad \st \quad (x,\eta,w,z) \in \Omega_j \subseteq \R^n \times \R_{\geq 0}
  \times \R_{\geq 0} \times \R^{n \times \Abs{\newUncertaintySet}}_{\geq 0},
\end{equation*}
where the set~$\Omega_j$ is obtained from
\begin{equation*}
  \Omega_0 \define \Set{(x,\eta,w,z) \in \bar X \times \R_{\geq 0} \times
    \R_{\geq 0} \times \R^{n \times \Abs{\newUncertaintySet}}_{\geq 0}}
\end{equation*}
by adding all valid inequalities that have been generated along the path from
the root to node~$j$ as well as all branching decisions made along that path.
If a solution~$(x^j,\eta^j,w^j,z^j)$ to the problem at node~$j$
satisfies~$\eta^j < R(x^j,\nu^*)$ for some $\nu^* \in \newUncertaintySet$, we
augment the set~$\Omega_j$ with~$n+1$ inequality constraints, namely
\begin{align*}
  \eta \geq \sum_{i=1}^n \nomCostFun{i}{\nu^*} x_i
  + \Gamma w^{\nu^*} + \sum_{i=1}^n z^{\nu^*}_i - k\nu^*, \quad
  w^{\nu^*} + z^{\nu^*}_i \geq \devCostFun{i}{\nu^*}
  x_i, \quad i \in N,
\end{align*}
in Line~\ref{alg:node-processing:projection:add-cut} of
Algorithm~\ref{alg:node-processing:projection}.


\section{Computational Results}
\label{sec:computational-results}

We now computationally assess and compare the performance of the
exact solution methods presented in Section~\ref{sec:solution-approaches}.
To this end, we brief\/ly describe the generation of the test instances and the
computational setup in Sections~\ref{sec:instances} and~\ref{sec:setup},
respectively.
%
In Sections~\ref{sec:comp-results:milp-bnc-ccg}
and~\ref{sec:comp-results:comparison}, we compare the solution methods
presented in this paper with respect to (i)~runtimes,
(ii)~the number of instances solved to global optimality, and
(iii)~optimality gaps.
Note that, to the best of our knowledge, there are currently no other methods
in the literature that tackle~$k$-delete recoverable robust~$0$--$1$
problems under budgeted uncertainty. Hence, there are no alternative methods
that we could compare with.
In addition to computational metrics, we discuss qualitative aspects 
of solutions in Section~\ref{sec:gain-n-price}. To this end, we consider the
\emph{gain of recovery}, i.e., the decrease in the optimal objective value
obtained by allowing for recovery actions.
Moreover, we elaborate on the benefits of using a robust model by
analyzing the reduction in the worst-case objective value achieved when
considering robust rather than nominal solutions.
Our main findings are summarized in
Figures~\ref{fig:projected-cuts}--\ref{fig:price-robust} and
Table~\ref{tab:comparison}. We provide supplementary results in
Appendix~\ref{sec:appendix-comp-results}.

\subsection{Generation of Test Instances}
\label{sec:instances}

In our computational study, we consider instances of two well-known
combinatorial optimization problems, both of which capture many
practical applications: the assignment problem and the single-source
capacitated facility location problem. Comprehensive overviews of these problem
classes can, e.g., be found in \textcite{Burkard_et_al:2012} and
\textcite{Saldanha-da-Gama_Wang:2024}.
The latter particularly addresses facility location problems under uncertainty.
In what follows, we describe the problem setting for both classes and explain
how we generate the corresponding data.

\subsubsection{Assignment Problem}

The assignment problem (AP) seeks to assign $n$~agents, i.e., staff,
machines, or resources, to~$n$ tasks so that each agent is assigned to exactly
one task and each task is assigned to exactly one agent.
Assigning agent~$i$ to task~$j$ incurs a cost of~$\cost_{ij} \in \R_{\geq 0}$,
and the objective is to determine an assignment that minimizes the total cost.
In its deterministic form, the assignment problem reads
\begin{align*}
  \min_x \quad
  & \sum_{i=1}^n \sum_{j=1}^n \cost_{ij}x_{ij}
  \\
  \st \quad
  & \sum_{i=1}^n x_{ij} = 1, \quad j \in [n],
  \\
  & \sum_{j=1}^n x_{ij} = 1, \quad i \in [n],
  \\
  & x_{ij} \in \Set{0,1}, \quad i,j \in [n].
\end{align*}
We generate $40$~deterministic instances of the assignment problem as follows.
Starting from the \textsf{assign800} instance of the \textsf{OR-Library}
\parencite{Beasley:1990,ORLIB:1990}, we randomly select a subset of tasks until
the desired number~$n$ is reached.
For each~$n \in \set{25, 50, 75, 100}$, we generate~$10$~instances.
To adapt these instances to a recoverable robust setup, we use a procedure
similar to that in \textcite{Buesing_et_al:2011a}, which works as follows.
For all~$i,j \in [n]$, we set the first-stage cost
to~$\initCost_{ij} = \lceil 0.6\cost_{ij} \rceil$ and the nominal cost
to~$\nomCost_{ij} = \lceil 0.2\cost_{ij} \rceil$.
The worst-case cost increase~$\devCost_{ij}$ is obtained by generating a
uniformly distributed random value~$\delta_{ij} \in [0.2, 0.4)$ and
setting~$\devCost_{ij} = \lceil \delta_{ij}\cost_{ij} \rceil$.
The robustness parameter is set to~$\Gamma = \lceil \gamma n \rceil$
with~$\gamma \in \set{0.1, 0.25, 0.5}$ and the recovery parameter
is set to~$k = \lceil \kappa n \rceil$ with~$\kappa \in \set{0.1, 0.25}$.
To sum up, we consider~$240$~instances of the~$k$-delete recoverable robust
assignment problem.

\subsubsection{Single-Source Capacitated Facility Location Problem}

Let~$I$ and~$J$ be given sets of customers and potential facility locations,
respectively. Each customer~$i \in I$ has a demand~$d_i \in \R_{\geq 0}$ that
needs to be served and each facility~$j \in J$ is associated with an opening
cost of~$b_j \in \R_{\geq}$ and a capacity of~$s_j \in \R_{\geq 0}$. Assigning
customer~$i$ to facility~$j$ incurs a cost of~$\cost_{ij} \in \R_{\geq 0}$.
Each customer must be assigned to one facility, which supplies its
entire demand, and the total demand served by a facility cannot
exceed its capacity. The single-source capacitated facility location problem
(SSCFLP) seeks to determine which facilities to open and how to assign
customers so that the total cost of opening facilities and serving
customers is minimized. The deterministic SSCFLP is given by
\begin{align*}
  \min_{x,y} \quad
  & \sum_{j \in J} b_jy_j + \sum_{i \in I} \sum_{j \in J} \cost_{ij}x_{ij}
  \\
  \st \quad
  & \sum_{j \in J} x_{ij} = 1, \quad i \in I,
  \\
  & x_{ij} \leq y_j, \quad i \in I,\, j \in J,
  \\
  & \sum_{i \in I} d_ix_{ij} \leq s_jy_j, \quad j \in J,
  \\
  & x_{ij}, y_j \in \Set{0,1}, \quad i \in I,\, j \in J.
\end{align*}
We consider the~$37$ deterministic facility location problem instances from the
\textsf{OR-Library} \parencite{Beasley:1990,ORLIB:1990} labeled
\textsf{cap41}--\textsf{cap134}, for which a summary is given in
Table~\ref{tab:SSCFLP-instances}.
\begin{table}
  \centering
  \caption{The number of customers (``$\Abs{I}$''), potential
    facility locations (``$\Abs{J}$''), and instances (``\# instances'') for
    \textsf{cap41}--\textsf{cap134}.}
  \begin{tabular}{lrrr}
    \toprule
    name & $\Abs{I}$ & $\Abs{J}$ & \# instances \\
    \midrule
    \textsf{cap41}--\textsf{cap74} & 50 & 16 & 13 \\
    \textsf{cap81}--\textsf{cap104} & 50 & 25 & 12 \\
    \textsf{cap121}--\textsf{cap134} & 50 & 50 & 12 \\
    \bottomrule
  \end{tabular}
  \label{tab:SSCFLP-instances}
\end{table}
Because these instances correspond to the (multi-source) capacitated facility
location problem, we adapted them to the single-source setting by requiring
that each customer is served by exactly one facility. To ensure feasibility, we
remove all customers whose demand exceeds the maximum capacity, i.e.,
we remove all~$i$ from~$I$ for which~$d_i > \max \defset{s_j}{j \in J}$
holds. We assume that only the assignment costs~$\cost$ are uncertain, whereas
the facility opening costs~$b$ remain deterministic.
The robustness parameter is set to~$\Gamma = \lceil \gamma \Abs{I} \rceil$
with~$\gamma \in \set{0.1, 0.25, 0.5}$ and the recovery parameter
is set to~$k = \lceil \kappa \Abs{I} \rceil$
with~$\kappa \in \set{0.1, 0.25}$.
All remaining data, i.e., $\initCost$, $\nomCost$, and~$\devCost$, is generated
in the same way as for the instances of the assignment problem. Overall, we
consider~$222$~instances of the~$k$-delete recoverable robust SSCFLP.

\subsection{Computational Setup}
\label{sec:setup}

All tests were carried out on an Intel XEON 8468 Sapphire at
\SI{2.1}{\giga\hertz} ($8$~cores) and \SI{32}{\giga\byte} RAM, which is
part of the High-Performance Computing cluster ``CLAIX'' at RWTH Aachen
University.
We implemented our solution approaches in \textsf{Python}~$3.12.3$ and use
\textsf{Gurobi}~12.0.0 to solve all arising optimization problems.
The code, along with the instance data used in our computational study,
is publicly available at
\url{https://github.com/YasmineBeck/k-delete-recoverable-robust-methods}.
A time limit of~\SI{1}{\hour} was set for solving each instance.
For the branch-and-cut methods, we add valid inequalities using
\textsf{Gurobi}'s lazy constraint callback, which requires setting the
parameter \textsf{LazyConstraints} to~$1$.
All other parameters were kept at their default settings.
We now elaborate on the scenario initialization and the full evaluation
frequency for the CCG methods as well as on the cut separation strategies used
in our branch-and-cut~approaches.

\subsubsection{Initializing the Scenario Set}
\label{sec:init-strategies}

For initializing the column-and-constraint-generation methods, we need to
specify the scenario sets~$\scenarios^0$ and~$\newUncertaintySet^0$ in
Algorithms~\ref{alg:ccg} and~\ref{alg:ccg:compact-ref}, respectively.
In preliminary computational experiments, we considered several ways to
initialize these sets as singletons, including greedy and randomized
approaches. These tests revealed that the impact of such scenario
initialization strategies is marginal.
In what follows, we thus only report the results for the setting in which the
CCG methods are initialized using the nominal scenario, i.e.,
$\scenarios^0 \define \defset{s=1}{\scenarioCost{s} = \nomCost}$
and~$\newUncertaintySet^0 \define \set{0}$.


\subsubsection{Full Evaluation Frequency}
\label{sec:full_eval}

As discussed in Section~\ref{sec:ccg:compact-formulation}, we aim for a
trade-off between the frequency at which we update the upper bound and the
computational effort of the CCG method by adjusting the full evaluation
frequency, i.e., the frequency at which we solve
Problem~\eqref{eq:adversarial-prob:compact} for
all~$\nu \in \newUncertaintySet$ in Algorithm~\ref{alg:ccg:compact-ref}.
In preliminary computational experiments, we compared two strategies:
performing a full evaluation (i)~in every iteration or (ii)~only
every~$10$th iteration. In the latter case, we terminate the for-loop in
Line~\ref{alg:ccg:compact-ref:for-loop} of Algorithm~\ref{alg:ccg:compact-ref}
as soon as a~$\nu \in \newUncertaintySet$ with~$\eta^j < R(x^j,\nu)$ is found.
Our tests revealed that the most suitable strategy strongly depends on the
problem at hand.
For AP, we observe that performing a full evaluation only
every~$10$th iteration leads to considerably shorter runtimes, whereas for
SSCFLP it seems beneficial to evaluate all~$\nu \in \newUncertaintySet$ in
every iteration.
In particular, this suggests that more involved problem classes (such as
SSCFLP compared to AP) tend to benefit from evaluating
all~$\nu \in \newUncertaintySet$ and potentially updating the upper bound more
frequently. However, if solving the master problem is computationally rather
cheap, this seems less advantageous.
In what follows, we thus perform a full evaluation in every iteration for
SSCFLP and only every~$10$th iteration for AP.

\subsubsection{Cut Separation Strategies}
\label{sec:cut-strategies}

As discussed in Section~\ref{sec:bnc}, cuts associated with multiple
uncertainty realizations~$\nu \in \newUncertaintySet$ may be added at each
node of the branch-and-cut search tree.
Nevertheless, different cut separation strategies can be used to reduce the
number of added cuts and, thus, potentially speed up the solution process.
We consider four such strategies.
\begin{description}[leftmargin=0em]
\item[\textsf{All-In}] Add the cut(s) for all~$\nu \in \newUncertaintySet$ whose
  constraints are violated by the solution to the current node problem.
\item[\textsf{First-In}] Iterate over~$\nu \in \newUncertaintySet$, add the
  cut(s) for the first uncertainty realization whose constraints are violated
  by the solution to the current node problem, and break the loop.
\item[\textsf{Shuffle-First-In}] Randomly shuffle the elements
  in~$\newUncertaintySet$ and apply \textsf{First-In}.
\item[\textsf{Max-Violation}] Add the cut(s) only for the uncertainty
  realization~$\nu \in \newUncertaintySet$ whose constraints are maximally
  violated by the solution to the current node problem.
\end{description}

In Figure~\ref{fig:projected-cuts}, we show empirical cumulative
distribution functions (ECDFs) of the runtimes and the optimality gaps
to compare the four considered cut separation strategies.
The ECDFs can be interpreted as the percentage of instances ($y$-axis) that can
be solved within a given time or with a given optimality gap ($x$-axis).
Here, we exemplarily focus on the branch-and-cut methods applied to the
projection-based formulation, but preliminary computational results revealed
that the same qualitative observations can also be made for the approaches
applied to the extended and the compact formulation.

\begin{figure}
  \centering
\begin{tikzpicture}[scale=0.75, baseline={(0,0)}]
  \begin{semilogxaxis}[
    xlabel = {Runtime (\si{\second})},
    xticklabels = {0.1,1,10,100,1000},
    y tick label style={
      /pgf/number format/.cd,
      fixed,
      fixed zerofill,
      precision=1,
      /tikz/.cd
    },
    xtick pos = bottom,
    ytick pos = left,
    xmin = 0.5668144226074219,
    xmax = 3579.0349221229553,
    ymin = 0,
    ymax = 1,
    legend pos=south east,
    legend cell align={left},
    xmajorgrids=true,
    ymajorgrids=true,
    grid style=dotted
    ]
    \addplot[my-red,very thick] table [x=data, y=percent, col sep=comma] {csv-files/assign/projected-runtime-0.csv};
    \addlegendentry{\textsf{All-In}}
    \addplot[my-blue,very thick,densely dashed] table [x=data, y=percent, col sep=comma] {csv-files/assign/projected-runtime-1.csv};
    \addlegendentry{\textsf{First-In}}
    \addplot[my-yellow,very thick,dashed] table [x=data, y=percent, col sep=comma] {csv-files/assign/projected-runtime-2.csv};
    \addlegendentry{\textsf{Shuffle-First-In}}
    \addplot[my-green,very thick,densely dotted] table [x=data, y=percent, col sep=comma] {csv-files/assign/projected-runtime-3.csv};
    \addlegendentry{\textsf{Max-Violation}}
  \end{semilogxaxis}
\end{tikzpicture}
\quad
\begin{tikzpicture}[scale=0.75, baseline={(0,0)}]
  \begin{axis}[
    xlabel = {Optimality Gap (in \si{\percent})},
    xticklabel = {\pgfmathparse{\tick*100}\pgfmathprintnumber{\pgfmathresult}},
    y tick label style={
      /pgf/number format/.cd,
      fixed,
      fixed zerofill,
      precision=1,
      /tikz/.cd
    },
    xtick pos = bottom,
    ytick pos = left,
    xmin = 0,
    xmax = 0.1701120071987725,
    ymin = 0,
    ymax = 1,
    legend pos=south east,
    legend cell align={left},
    xmajorgrids=true,
    ymajorgrids=true,
    grid style=dotted
    ]
    \addplot[my-red,very thick] table [x=data, y=percent, col sep=comma] {csv-files/assign/projected-gap-0.csv};
    \addlegendentry{\textsf{All-In}}
    \addplot[my-blue,very thick,densely dashed] table [x=data, y=percent, col sep=comma] {csv-files/assign/projected-gap-1.csv};
    \addlegendentry{\textsf{First-In}}
    \addplot[my-yellow,very thick,dashed] table [x=data, y=percent, col sep=comma] {csv-files/assign/projected-gap-2.csv};
    \addlegendentry{\textsf{Shuffle-First-In}}
    \addplot[my-green,very thick,densely dotted] table [x=data, y=percent, col sep=comma] {csv-files/assign/projected-gap-3.csv};
    \addlegendentry{\textsf{Max-Violation}}
  \end{axis}
\end{tikzpicture}
\qquad
\begin{tikzpicture}[scale=0.75, baseline={(0,0)}]
  \begin{semilogxaxis}[
    xlabel = {Runtime (\si{\second})},
    xticklabels = {0.1,1,10,100,1000},
    y tick label style={
      /pgf/number format/.cd,
      fixed,
      fixed zerofill,
      precision=1,
      /tikz/.cd
    },
    xtick pos = bottom,
    ytick pos = left,
    xmin = 0.7474339008331299,
    xmax = 3539.8522148132324,
    ymin = 0,
    ymax = 1,
    legend pos=north west,
    legend cell align={left},
    xmajorgrids=true,
    ymajorgrids=true,
    grid style=dotted
    ]
    \addplot[my-red,very thick] table [x=data, y=percent, col sep=comma] {csv-files/cap/projected-runtime-0.csv};
    \addlegendentry{\textsf{All-In}}
    \addplot[my-blue,very thick,densely dashed] table [x=data, y=percent, col sep=comma] {csv-files/cap/projected-runtime-1.csv};
    \addlegendentry{\textsf{First-In}}
    \addplot[my-yellow,very thick,dashed] table [x=data, y=percent, col sep=comma] {csv-files/cap/projected-runtime-2.csv};
    \addlegendentry{\textsf{Shuffle-First-In}}
    \addplot[my-green,very thick,densely dotted] table [x=data, y=percent, col sep=comma] {csv-files/cap/projected-runtime-3.csv};
    \addlegendentry{\textsf{Max-Violation}}
  \end{semilogxaxis}
\end{tikzpicture}
\quad
\begin{tikzpicture}[scale=0.75, baseline={(0,0)}]
  \begin{axis}[
    xlabel = {Optimality Gap (in \si{\percent})},
    xticklabel = {\pgfmathparse{\tick*100}\pgfmathprintnumber{\pgfmathresult}},
    y tick label style={
      /pgf/number format/.cd,
      fixed,
      fixed zerofill,
      precision=1,
      /tikz/.cd
    },
    xtick pos = bottom,
    ytick pos = left,
    xmin = 0,
    xmax = 0.5544816398832612,
    ymin = 0,
    ymax = 1,
    legend pos=south east,
    legend cell align={left},
    xmajorgrids=true,
    ymajorgrids=true,
    grid style=dotted
    ]
    \addplot[my-red,very thick] table [x=data, y=percent, col sep=comma] {csv-files/cap/projected-gap-0.csv};
    \addlegendentry{\textsf{All-In}}
    \addplot[my-blue,very thick,densely dashed] table [x=data, y=percent, col sep=comma] {csv-files/cap/projected-gap-1.csv};
    \addlegendentry{\textsf{First-In}}
    \addplot[my-yellow,very thick,dashed] table [x=data, y=percent, col sep=comma] {csv-files/cap/projected-gap-2.csv};
    \addlegendentry{\textsf{Shuffle-First-In}}
    \addplot[my-green,very thick,densely dotted] table [x=data, y=percent, col sep=comma] {csv-files/cap/projected-gap-3.csv};
    \addlegendentry{\textsf{Max-Violation}}
  \end{axis}
\end{tikzpicture}
  \caption{Log-scaled ECDFs of the runtimes (in \si{\second}) and
    linear-scaled ECDFs of the optimality gaps (in \si{\percent}) for the
    branch-and-cut approaches with different cut separation strategies applied
    to the projection-based formulation.
    Results for AP are shown in the top figures and results for
    SSCFLP are shown in the bottom figures.}
  \label{fig:projected-cuts}
\end{figure}
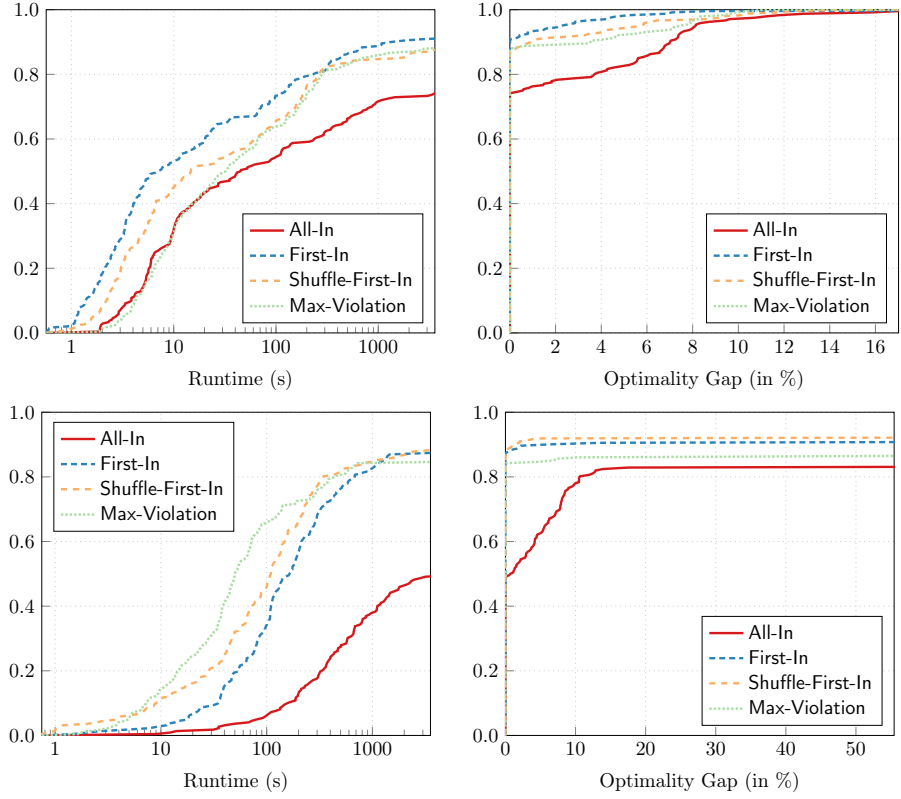

For both AP and SSCFLP, we observe that adding only one cut per node, i.e.,
using strategies \textsf{First-In}, \textsf{Shuffle-First-In}, or
\textsf{Max-Violation}, significantly improves the performance of the overall
branch-and-cut method in terms of running times and, consequently, the number
of instances solved to global optimality; see also
Table~\ref{tab:projected-cuts} in Appendix~\ref{sec:appendix-comp-results}.
Although the most effective choice among the considered strategies seems to
depend on the specific problem under consideration, we observe that
\textsf{Shuffle-First-In} consistently performs well for both the assignment
problem and the single-source capacitated facility location problem.
This can be explained by the fact that this cut separation strategy is
computationally rather cheap and that randomization helps to avoid
repeatedly adding cuts associated with the same scenario in consecutive
iterations, resulting in a more balanced exploration of the scenario set.
The previous observations align particularly well with the findings of the
computational studies in \textcite{Beck_et_al:2023,Beck:2024}, in which
bilevel problems with~$\Gamma$-robust followers are studied. As supported by the
results in \textcite{Goerigk_et_al:2025}, such robust bilevel problems exhibit a
multi-level structure similar to that of the~$k$-delete recoverable robust problem
studied in this work.
To sum up, we thus fix the cut separation strategy to \textsf{Shuffle-First-In}
for all branch-and-cut approaches in the following.

\subsection{MILP vs.\ BnC vs.\ CCG}
\label{sec:comp-results:milp-bnc-ccg}

Whereas only a single method (CCG or BnC) is available for solving the
scenario-based formulation~\eqref{eq:scenario-ref} and the projection-based
formulation~\eqref{eq:proj-ref} of
Problem~\eqref{eq:recoverable-robust-problem}, we can choose
among three methods---the MILP, the branch-and-cut, and the CCG approach---for
the extended formulation~\eqref{eq:extended-formulation} and the compact
formulation~\eqref{eq:compact-ref}; cf.\ Table~\ref{tab:approaches}.
In what follows, we compare these three approaches for both formulations.

\subsubsection{Extended Formulation~\eqref{eq:extended-formulation}}
\label{sec:comp-results:extended}

\begin{figure}
  \centering
\begin{tikzpicture}[scale=0.75, baseline={(0,0)}]
  \begin{semilogxaxis}[
    xlabel = {Runtime (\si{\second})},
    xticklabels = {0.1,1,10,100,1000},
    y tick label style={
      /pgf/number format/.cd,
      fixed,
      fixed zerofill,
      precision=1,
      /tikz/.cd
    },
    xtick pos = bottom,
    ytick pos = left,
    xmin = 0.5728592872619629,
    xmax = 3304.0190846920013,
    ymin = 0,
    ymax = 1,
    legend pos=south east,
    legend cell align={left},
    xmajorgrids=true,
    ymajorgrids=true,
    grid style=dotted
    ]
    \addplot[my-red,very thick] table [x=data, y=percent, col sep=comma] {csv-files/assign/extended-runtime-0.csv};
    \addlegendentry{{MILP}}
    \addplot[my-blue,very thick,densely dashed] table [x=data, y=percent, col sep=comma] {csv-files/assign/extended-runtime-1.csv};
    \addlegendentry{{BnC}}
    \addplot[my-yellow,very thick,dashed] table [x=data, y=percent, col sep=comma] {csv-files/assign/extended-runtime-2.csv};
    \addlegendentry{{CCG}}
  \end{semilogxaxis}
\end{tikzpicture}
\quad
\begin{tikzpicture}[scale=0.75, baseline={(0,0)}]
  \begin{axis}[
    xlabel = {Optimality Gap (in \si{\percent})},
    scaled x ticks = false,
    xticklabel = {\pgfmathparse{\tick*100}\pgfmathprintnumber{\pgfmathresult}},
    y tick label style={
      /pgf/number format/.cd,
      fixed,
      fixed zerofill,
      precision=1,
      /tikz/.cd
    },
    xtick pos = bottom,
    ytick pos = left,
    xmin = 0,
    xmax = 0.025579536370903187,
    ymin = 0,
    ymax = 1,
    legend pos=south east,
    legend cell align={left},
    xmajorgrids=true,
    ymajorgrids=true,
    grid style=dotted
    ]
    \addplot[my-red,very thick] table [x=data, y=percent, col sep=comma] {csv-files/assign/extended-gap-0.csv};
    \addlegendentry{{MILP}}
    \addplot[my-blue,very thick,densely dashed] table [x=data, y=percent, col sep=comma] {csv-files/assign/extended-gap-1.csv};
    \addlegendentry{{BnC}}
    \addplot[my-yellow,very thick,dashed] table [x=data, y=percent, col sep=comma] {csv-files/assign/extended-gap-2.csv};
    \addlegendentry{{CCG}}
  \end{axis}
\end{tikzpicture}
\qquad
\begin{tikzpicture}[scale=0.75, baseline={(0,0)}]
  \begin{semilogxaxis}[
    xlabel = {Runtime (\si{\second})},
    xticklabels = {1,10,100,1000},
    y tick label style={
      /pgf/number format/.cd,
      fixed,
      fixed zerofill,
      precision=1,
      /tikz/.cd
    },
    xtick pos = bottom,
    ytick pos = left,
    xmin = 5.202984094619751,
    xmax = 3663.0492911338806,
    ymin = 0,
    ymax = 1,
    legend pos=north west,
    legend cell align={left},
    xmajorgrids=true,
    ymajorgrids=true,
    grid style=dotted
    ]
    \addplot[my-red,very thick] table [x=data, y=percent, col sep=comma] {csv-files/cap/extended-runtime-0.csv};
    \addlegendentry{{MILP}}
    \addplot[my-blue,very thick,densely dashed] table [x=data, y=percent, col sep=comma] {csv-files/cap/extended-runtime-1.csv};
    \addlegendentry{{BnC}}
    \addplot[my-yellow,very thick,dashed] table [x=data, y=percent, col sep=comma] {csv-files/cap/extended-runtime-2.csv};
    \addlegendentry{{CCG}}
  \end{semilogxaxis}
\end{tikzpicture}
\quad
\begin{tikzpicture}[scale=0.75, baseline={(0,0)}]
  \begin{axis}[
    xlabel = {Optimality Gap (in \si{\percent})},
    scaled x ticks = false,
    xticklabel = {\pgfmathparse{\tick*100}\pgfmathprintnumber{\pgfmathresult}},
    y tick label style={
      /pgf/number format/.cd,
      fixed,
      fixed zerofill,
      precision=1,
      /tikz/.cd
    },
    xtick pos = bottom,
    ytick pos = left,
    xmin = 0,
    xmax = 0.2628029864610451,
    ymin = 0,
    ymax = 1,
    legend pos=south east,
    legend cell align={left},
    xmajorgrids=true,
    ymajorgrids=true,
    grid style=dotted
    ]
    \addplot[my-red,very thick] table [x=data, y=percent, col sep=comma] {csv-files/cap/extended-gap-0.csv};
    \addlegendentry{{MILP}}
    \addplot[my-blue,very thick,densely dashed] table [x=data, y=percent, col sep=comma] {csv-files/cap/extended-gap-1.csv};
    \addlegendentry{{BnC}}
    \addplot[my-yellow,very thick,dashed] table [x=data, y=percent, col sep=comma] {csv-files/cap/extended-gap-2.csv};
    \addlegendentry{{CCG}}
  \end{axis}
\end{tikzpicture}
  \caption{Log-scaled ECDFs of the runtimes (in \si{\second}) and
    linear-scaled ECDFs of the optimality gaps (in \si{\percent}) for the
    approaches MILP, BnC, and CCG applied to the
    extended formulation~\eqref{eq:extended-formulation}.
    Results for AP are shown in the top figures and results for
    SSCFLP are shown in the bottom figures.}
  \label{fig:extended}
\end{figure}
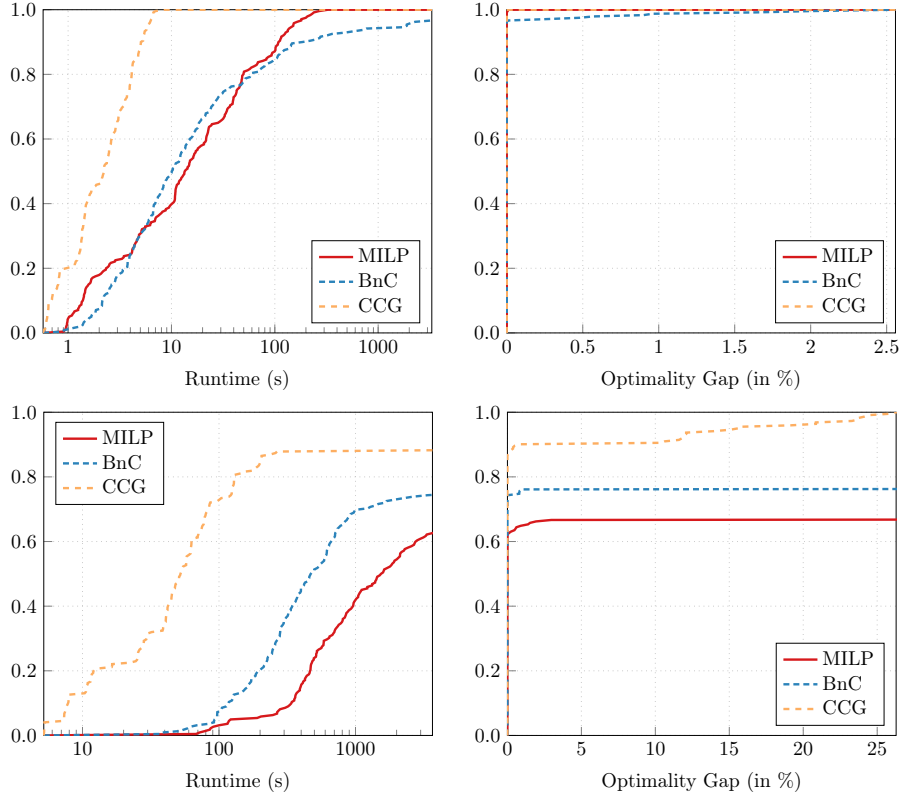

In Figure~\ref{fig:extended}, we show ECDFs of the runtimes and optimality gaps
for the three approaches applied to the extended formulation.
For AP, we observe that the CCG method and the MILP approach can solve
all~$240$~considered instances to global optimality within the time limit
of~\SI{1}{\hour}, whereas the BnC approach solves~$232$~instances
(\SI{96.67}{\percent}). Nevertheless, BnC finds a feasible point with finite
optimality gap for all instances, with the largest optimality
gap observed being~\SI{2.56}{\percent}. Regarding runtimes, we emphasize that
the CCG method clearly outperforms the other two approaches.
The same qualitative behavior can also be observed for the SSCFLP
instances. Again, CCG performs significantly better compared to the
branch-and-cut and the MILP approach.
As a result, more instances are solved to global optimality using CCG and, in
particular, a feasible point with finite optimality gap is found for
all~$222$~SSCFLP instances.
In contrast, the MILP approach and BnC find feasible points with
finite optimality gaps for only \SI{66.67}{\percent} and \SI{76.13}{\percent}
of these instances, respectively.
These observations are further supported by the results shown in
Table~\ref{tab:extended} in Appendix~\ref{sec:appendix-comp-results}. Overall,
applying CCG to the extended formulation thus seems to be the most effective
approach.

\subsubsection{Compact Formulation~\eqref{eq:compact-ref}}
\label{sec:comp-results:compact}

For the assignment problem, the previous observations regarding the
performance of the MILP, BnC, and CCG approaches are even more pronounced
when considering the compact formulation.
In Figure~\ref{fig:compact}, we show ECDFs of the runtimes and the optimality
gaps for the three approaches.
It can again be seen that the CCG method performs considerably
better than the other two approaches both in terms of running times and
solution quality. In particular, because CCG can solve the AP instances more
quickly, a larger number of instances can be solved to global optimality.
Overall, $238$ out of~$240$~instances (\SI{99.17}{\percent}) can be solved
using CCG within the time limit of~\SI{1}{\hour}, whereas no feasible point can
be computed for the two remaining instances.
In contrast, a feasible point with finite optimality gap is found for only
\SI{25}{\percent} and \SI{88.75}{\percent} of the instances using the MILP and
the BnC approaches, respectively.
The poor performance of the MILP approach for the compact formulation can be
explained by the fact that computing the modified cost
coefficients~$\compCostFun{i}{\nu}{\ell}$ for all~$\ell \in \mathcal{L}$,
$\nu \in \newUncertaintySet$, and~$i \in N$ is extremely expensive.
In fact, this computation consumes the majority of both runtime and
memory. Consequently, for many instances considered in this computational
study, the model cannot be built due to either time or memory constraints.

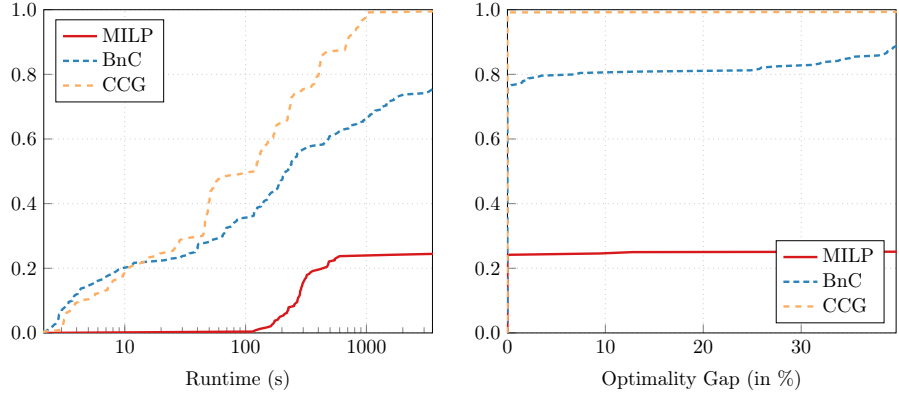
\begin{figure}
  \centering
\begin{tikzpicture}[scale=0.75, baseline={(0,0)}]
  \begin{semilogxaxis}[
    xlabel = {Runtime (\si{\second})},
    xticklabels = {1,10,100,1000},
    y tick label style={
      /pgf/number format/.cd,
      fixed,
      fixed zerofill,
      precision=1,
      /tikz/.cd
    },
    xtick pos = bottom,
    ytick pos = left,
    xmin = 2.139267921447754,
    xmax = 3518.847832918167,
    ymin = 0,
    ymax = 1,
    legend pos=north west,
    legend cell align={left},
    xmajorgrids=true,
    ymajorgrids=true,
    grid style=dotted
    ]
    \addplot[my-red,very thick] table [x=data, y=percent, col sep=comma] {csv-files/assign/compact-runtime-0.csv};
    \addlegendentry{{MILP}}
    \addplot[my-blue,very thick,densely dashed] table [x=data, y=percent, col sep=comma] {csv-files/assign/compact-runtime-1.csv};
    \addlegendentry{{BnC}}
    \addplot[my-yellow,very thick,dashed] table [x=data, y=percent, col sep=comma] {csv-files/assign/compact-runtime-2.csv};
    \addlegendentry{{CCG}}
  \end{semilogxaxis}
\end{tikzpicture}
\quad
\begin{tikzpicture}[scale=0.75, baseline={(0,0)}]
  \begin{axis}[
    xlabel = {Optimality Gap (in \si{\percent})},
    scaled x ticks = false,
    xticklabel = {\pgfmathparse{\tick*100}\pgfmathprintnumber{\pgfmathresult}},
    y tick label style={
      /pgf/number format/.cd,
      fixed,
      fixed zerofill,
      precision=1,
      /tikz/.cd
    },
    xtick pos = bottom,
    ytick pos = left,
    xmin = 0,
    xmax = 0.3969465648854962,
    ymin = 0,
    ymax = 1,
    legend pos=south east,
    legend cell align={left},
    xmajorgrids=true,
    ymajorgrids=true,
    grid style=dotted
    ]
    \addplot[my-red,very thick] table [x=data, y=percent, col sep=comma] {csv-files/assign/compact-gap-0.csv};
    \addlegendentry{{MILP}}
    \addplot[my-blue,very thick,densely dashed] table [x=data, y=percent, col sep=comma] {csv-files/assign/compact-gap-1.csv};
    \addlegendentry{{BnC}}
    \addplot[my-yellow,very thick,dashed] table [x=data, y=percent, col sep=comma] {csv-files/assign/compact-gap-2.csv};
    \addlegendentry{{CCG}}
  \end{axis}
\end{tikzpicture}
  \caption{Log-scaled ECDFs of the runtimes (in \si{\second}) and
    linear-scaled ECDFs of the optimality gaps (in \si{\percent}) for the
    approaches MILP, BnC, and CCG applied to the compact
    formulation~\eqref{eq:compact-ref} of~AP.}
  \label{fig:compact}
\end{figure}

For the same reasons, no incumbent solution is found for any instance of
the single-source capacitated facility location problem using the
MILP or CCG approaches. Therefore, we do not show ECDFs for the
compact formulation applied to SSCFLP. Nevertheless, the branch-and-cut approach
solves~$100$ out of~$222$~instances (\SI{45.05}{\percent}) to global optimality
within the time limit of~\SI{1}{\hour}. For additional~$72$~instances, a
feasible point with an open but finite gap of at most~\SI{10.86}{\percent} is
found, whereas no incumbent solution is found for the remaining~$50$~instances.
All previous observations are also supported by the results shown in
Table~\ref{tab:compact} in Appendix~\ref{sec:appendix-comp-results}.

To sum up, the CCG method outperforms the MILP and BnC approach for AP, whereas
branch-and-cut is the only approach capable of tackling instances of SSCFLP
when considering the compact formulation.

\subsection{Comparison of Formulations}
\label{sec:comp-results:comparison}

We now compare the four considered formulations---the scenario-based,
projection-based, extended, and compact formulation---by considering their
respective best-performing solution methods. This means that we apply CCG to
the scenario-based and the extended formulation, and we apply BnC to the
projection-based formulation. For the compact formulation, we make the
following distinction: we use CCG for AP and BnC for SSCFLP.

In Figure~\ref{fig:comparison}, we show ECDFs of the runtimes and the
optimality gaps for the approaches considered with these four formulations.
Moreover, we report the number of instances solved to global optimality in
Table~\ref{tab:comparison}.
For AP, it can be seen that the only approach capable of solving all of
the~$240$~considered instances to global optimality is the CCG approach applied
to the extended formulation. In particular, we observe that this approach also
outperforms all remaining ones in terms of running times.
For SSCFLP, we observe that the extended formulation again yields an approach
that performs considerably better than the remaining ones in terms of
runtimes. However, when considering the ECDFs of the optimality gaps in
Figure~\ref{fig:comparison}, it can be seen that the solution quality obtained
from the scenario-based formulation is slightly better. Nevertheless, the CCG
method applied to the extended formulation
significantly outperforms this approach in terms of running times. We also
point out that BnC for the compact formulation performs similarly well compared
to the CCG approach applied to the extended
formulation. However, as more instances can be solved to global optimality
using the extended formulation, we consider the latter the overall best
formulation for solving SSCFLP instances.

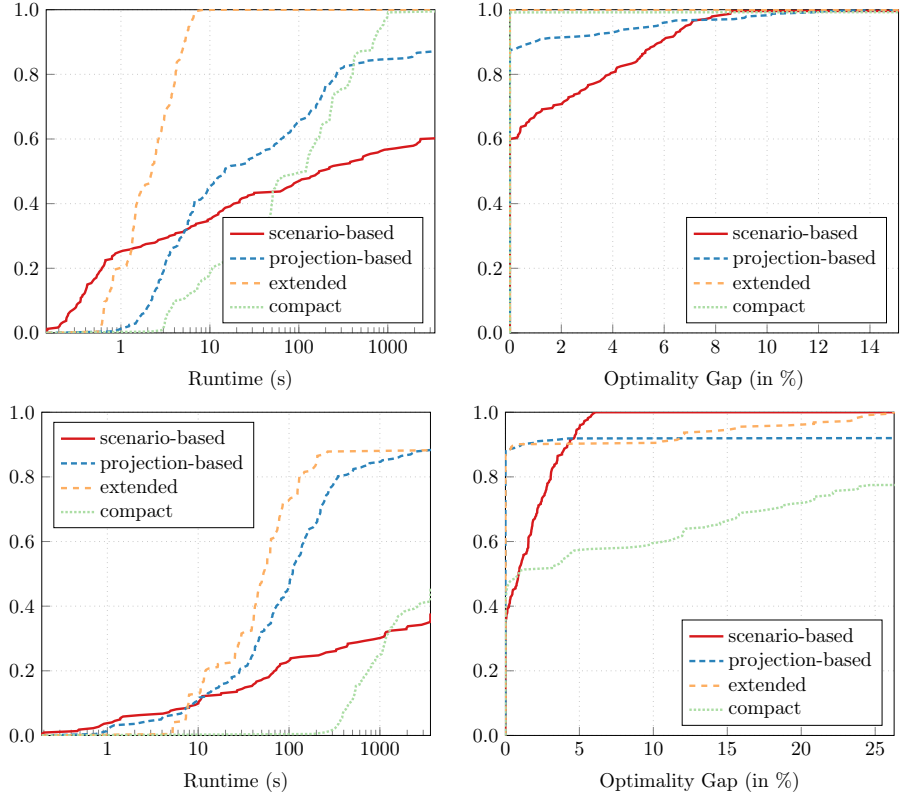
\begin{figure}
  \centering
\begin{tikzpicture}[scale=0.75, baseline={(0,0)}]
  \begin{semilogxaxis}[
    xlabel = {Runtime (\si{\second})},
    xticklabels = {0.1,1,10,100,1000},
    y tick label style={
      /pgf/number format/.cd,
      fixed,
      fixed zerofill,
      precision=1,
      /tikz/.cd
    },
    xtick pos = bottom,
    ytick pos = left,
    xmin = 0.1444711685180664,
    xmax = 3386.7642047405243,
    ymin = 0,
    ymax = 1,
    legend pos=south east,
    legend cell align={left},
    xmajorgrids=true,
    ymajorgrids=true,
    grid style=dotted
    ]
    \addplot[my-red,very thick] table [x=data, y=percent, col sep=comma] {csv-files/assign/winner-runtime-0.csv};
    \addlegendentry{{scenario-based}}
    \addplot[my-blue,very thick,densely dashed] table [x=data, y=percent, col sep=comma] {csv-files/assign/winner-runtime-1.csv};
    \addlegendentry{{projection-based}}
    \addplot[my-yellow,very thick,dashed] table [x=data, y=percent, col sep=comma] {csv-files/assign/winner-runtime-2.csv};
    \addlegendentry{{extended}}
    \addplot[my-green,very thick,densely dotted] table [x=data, y=percent, col sep=comma] {csv-files/assign/winner-runtime-3.csv};
    \addlegendentry{{compact}}
  \end{semilogxaxis}
\end{tikzpicture}
\quad
\begin{tikzpicture}[scale=0.75, baseline={(0,0)}]
  \begin{axis}[
    xlabel = {Optimality Gap (in \si{\percent})},
    xticklabel = {\pgfmathparse{\tick*100}\pgfmathprintnumber{\pgfmathresult}},
    y tick label style={
      /pgf/number format/.cd,
      fixed,
      fixed zerofill,
      precision=1,
      /tikz/.cd
    },
    xtick pos = bottom,
    ytick pos = left,
    xmin = 0,
    xmax = 0.15111679287698998,
    ymin = 0,
    ymax = 1,
    legend pos=south east,
    legend cell align={left},
    xmajorgrids=true,
    ymajorgrids=true,
    grid style=dotted
    ]
    \addplot[my-red,very thick] table [x=data, y=percent, col sep=comma] {csv-files/assign/winner-gap-0.csv};
    \addlegendentry{{scenario-based}}
    \addplot[my-blue,very thick,densely dashed] table [x=data, y=percent, col sep=comma] {csv-files/assign/winner-gap-1.csv};
    \addlegendentry{{projection-based}}
    \addplot[my-yellow,very thick,dashed] table [x=data, y=percent, col sep=comma] {csv-files/assign/winner-gap-2.csv};
    \addlegendentry{{extended}}
    \addplot[my-green,very thick,densely dotted] table [x=data, y=percent, col sep=comma] {csv-files/assign/winner-gap-3.csv};
    \addlegendentry{{compact}}
  \end{axis}
\end{tikzpicture}
\qquad
\begin{tikzpicture}[scale=0.75, baseline={(0,0)}]
  \begin{semilogxaxis}[
    xlabel = {Runtime (\si{\second})},
    xticklabels = {0.1,1,10,100,1000},
    y tick label style={
      /pgf/number format/.cd,
      fixed,
      fixed zerofill,
      precision=1,
      /tikz/.cd
    },
    xtick pos = bottom,
    ytick pos = left,
    xmin = 0.18972086906433105,
    xmax = 3605.4541091918945,
    ymin = 0,
    ymax = 1,
    legend pos=north west,
    legend cell align={left},
    xmajorgrids=true,
    ymajorgrids=true,
    grid style=dotted
    ]
    \addplot[my-red,very thick] table [x=data, y=percent, col sep=comma] {csv-files/cap/winner-runtime-0.csv};
    \addlegendentry{{scenario-based}}
    \addplot[my-blue,very thick,densely dashed] table [x=data, y=percent, col sep=comma] {csv-files/cap/winner-runtime-1.csv};
    \addlegendentry{{projection-based}}
    \addplot[my-yellow,very thick,dashed] table [x=data, y=percent, col sep=comma] {csv-files/cap/winner-runtime-2.csv};
    \addlegendentry{{extended}}
    \addplot[my-green,very thick,densely dotted] table [x=data, y=percent, col sep=comma] {csv-files/cap/winner-runtime-3.csv};
    \addlegendentry{{compact}}
  \end{semilogxaxis}
\end{tikzpicture}
\quad
\begin{tikzpicture}[scale=0.75, baseline={(0,0)}]
  \begin{axis}[
    xlabel = {Optimality Gap (in \si{\percent})},
    xticklabel = {\pgfmathparse{\tick*100}\pgfmathprintnumber{\pgfmathresult}},
    y tick label style={
      /pgf/number format/.cd,
      fixed,
      fixed zerofill,
      precision=1,
      /tikz/.cd
    },
    xtick pos = bottom,
    ytick pos = left,
    xmin = 0,
    xmax = 0.2628029864610451,
    ymin = 0,
    ymax = 1,
    legend pos=south east,
    legend cell align={left},
    xmajorgrids=true,
    ymajorgrids=true,
    grid style=dotted
    ]
    \addplot[my-red,very thick] table [x=data, y=percent, col sep=comma] {csv-files/cap/winner-gap-0.csv};
    \addlegendentry{{scenario-based}}
    \addplot[my-blue,very thick,densely dashed] table [x=data, y=percent, col sep=comma] {csv-files/cap/winner-gap-1.csv};
    \addlegendentry{{projection-based}}
    \addplot[my-yellow,very thick,dashed] table [x=data, y=percent, col sep=comma] {csv-files/cap/winner-gap-2.csv};
    \addlegendentry{{extended}}
    \addplot[my-green,very thick,densely dotted] table [x=data, y=percent, col sep=comma] {csv-files/cap/winner-gap-3.csv};
    \addlegendentry{{compact}}
  \end{axis}
\end{tikzpicture}
  \caption{Log-scaled ECDFs of the runtimes (in \si{\second}) and
    linear-scaled ECDFs of the optimality gaps (in \si{\percent}) for the
    best-performing approaches applied to the scenario-based, projection-based,
    extended, and compact formulations of
    Problem~\eqref{eq:recoverable-robust-problem}.
    Results for AP are shown in the top figures and results for
    SSCFLP are shown in the bottom figures.}
  \label{fig:comparison}
\end{figure}

\begin{table}
  \caption{The number of instances solved to global optimality (``solved''; out
    of $240$ and $222$ instances for AP and SSCFLP,
    respectively), the number of instances for which a
    feasible point with finite but non-zero gap is found (``open gap''), and
    the number of instances with infinite gap (``infinite gap'') for ``winner
    approaches'' of the four considered formulations.
    For those instances with finite but non-zero gap, also the average
    gap (``average gap''; in \si{\percent}) is~shown.}
  \begin{tabular}{llrrrr}
    \toprule
    problem & formulation & solved & open gap & average gap & infinite gap \\
    \midrule
    AP & 
    scenario-based & 144 & 95 & 3.88 & 1 \\
    & projection-based & 210 & 30 & 5.19 & 0 \\
    & extended & 240 & 0 & -- & 0 \\
    & compact & 238 & 0 & -- & 2 \\
    \midrule
    SSCFLP &
    scenario-based & 83 & 139 & 2.40 & 0 \\
    & projection-based & 196 & 8 & 1.87 & 18 \\
    & extended & 195 & 27 & 14.35 & 0 \\
    & compact & 100 & 72 & 10.86 & 50 \\
    \bottomrule
  \end{tabular}
  \label{tab:comparison}
\end{table}

Nevertheless, let us also compare our methods with respect to so-called
idealized parallel runtimes, which reflect the runtime of a method assuming
that sufficient capacities are available to solve all arising sub-problems in
parallel. As discussed in Section~\ref{sec:solution-approaches}, some of our
methods can be partially parallelized. Hence, for each instance, we compute the
idealized parallel runtime by first solving all sub-problems sequentially and
then taking the maximum runtime across all sub-problems. If an instance cannot
be solved within the time limit in the sequential setting, it is thus also
considered as unsolved in the idealized parallel setting.
In Figure~\ref{fig:comparison-idealized}, we show ECDFs of the idealized parallel
runtimes. It can be seen that the previous observations are even more
pronounced in this setting. Overall, applying the CCG method to the extended
formulation yields an approach that clearly outperforms the remaining ones in
terms of idealized parallel runtimes.
To sum up, the extended formulation~\eqref{eq:extended-formulation} thus seems
to be the most effective formulation for solving the~$k$-delete recoverable
robust $0$--$1$ problem.

\begin{figure}
  \centering
\begin{tikzpicture}[scale=0.75, baseline={(0,0)}]
  \begin{semilogxaxis}[
    xlabel = {Idealized Parallel Runtime (\si{\second})},
    xticklabels = {0,0.1,1,10,100,1000},
    y tick label style={
      /pgf/number format/.cd,
      fixed,
      fixed zerofill,
      precision=1,
      /tikz/.cd
    },
    xtick pos = bottom,
    ytick pos = left,
    xmin = 0.09515810012817383,
    xmax = 3499.321398973465,
    ymin = 0,
    ymax = 1,
    legend pos=south east,
    legend cell align={left},
    xmajorgrids=true,
    ymajorgrids=true,
    grid style=dotted
    ]
    \addplot[my-red,very thick] table [x=data, y=percent, col sep=comma] {csv-files/assign/winner-idealruntime-0.csv};
    \addlegendentry{{scenario-based}}
    \addplot[my-blue,very thick,densely dashed] table [x=data, y=percent, col sep=comma] {csv-files/assign/winner-idealruntime-1.csv};
    \addlegendentry{{projection-based}}
    \addplot[my-yellow,very thick,dashed] table [x=data, y=percent, col sep=comma] {csv-files/assign/winner-idealruntime-2.csv};
    \addlegendentry{{extended}}
    \addplot[my-green,very thick,densely dotted] table [x=data, y=percent, col sep=comma] {csv-files/assign/winner-idealruntime-3.csv};
    \addlegendentry{{compact}}
  \end{semilogxaxis}
\end{tikzpicture}
\quad
\begin{tikzpicture}[scale=0.75, baseline={(0,0)}]
  \begin{semilogxaxis}[
    xlabel = {Idealized Parallel Runtime (\si{\second})},
    xticklabels = {0,0.1,1,10,100,1000},
    y tick label style={
      /pgf/number format/.cd,
      fixed,
      fixed zerofill,
      precision=1,
      /tikz/.cd
    },
    xtick pos = bottom,
    ytick pos = left,
    xmin = 0.06347084045410156,
    xmax = 3600.04319691658,
    ymin = 0,
    ymax = 1,
    legend pos=north west,
    legend cell align={left},
    xmajorgrids=true,
    ymajorgrids=true,
    grid style=dotted
    ]
    \addplot[my-red,very thick] table [x=data, y=percent, col sep=comma] {csv-files/cap/winner-idealruntime-0.csv};
    \addlegendentry{{scenario-based}}
    \addplot[my-blue,very thick,densely dashed] table [x=data, y=percent, col sep=comma] {csv-files/cap/winner-idealruntime-1.csv};
    \addlegendentry{{projection-based}}
    \addplot[my-yellow,very thick,dashed] table [x=data, y=percent, col sep=comma] {csv-files/cap/winner-idealruntime-2.csv};
    \addlegendentry{{extended}}
    \addplot[my-green,very thick,densely dotted] table [x=data, y=percent, col sep=comma] {csv-files/cap/winner-idealruntime-3.csv};
    \addlegendentry{{compact}}
  \end{semilogxaxis}
\end{tikzpicture}
  \caption{Log-scaled ECDFs of the idealized parallel runtimes (in
    \si{\second}) for the best-performing approaches applied to the
    scenario-based, projection-based, extended, and compact formulations of
    Problem~\eqref{eq:recoverable-robust-problem}.
    Results for AP are shown on the right, whereas results for
    SSCFLP are shown on the left.}
  \label{fig:comparison-idealized}
\end{figure}
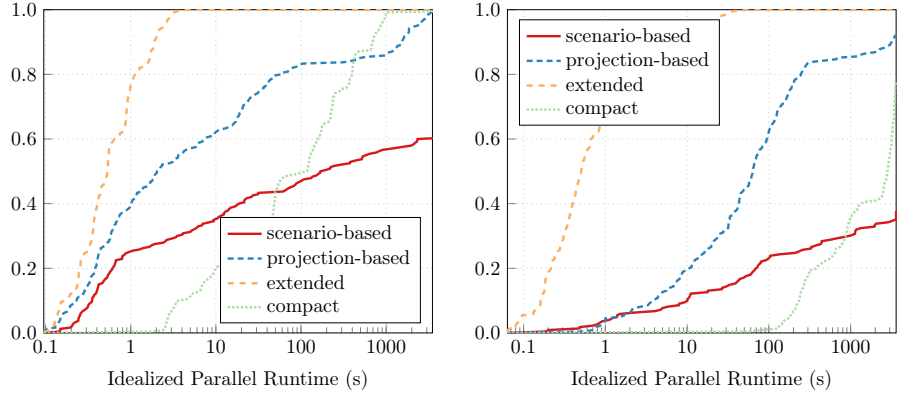

\subsection{Gain of Recovery and Impact of Robust Modeling}
\label{sec:gain-n-price}

To conclude this computational study, let us now discuss some qualitative
aspects of solutions.
We start with evaluating the gain of recovery, i.e., the decrease in the
optimal objective value obtained by incorporating the possibility of recovery
actions in the modeling. To this end, we compare the optimal objective value of
the~$k$-delete recoverable robust problem with the objective value obtained by
first solving the robust problem without recovery (i.e.,
Problem~\eqref{eq:recoverable-robust-problem} with~$k=0$) and then deleting up
to~$k$ components of the optimal solution ex post for which the worst-case cost
increase is realized.
In Figure~\ref{fig:gain-recovery}, we show boxplots for the gain of recovery
for the different choices of the recovery parameter~$k$. Note that we only
include instances that have been solved to global optimality.
In particular, our comparisons are based on the results obtained from the CCG
method applied to the extended formulation, which is the most effective
solution approach for both AP and SSCFLP instances, in order to include as many
instances solved to global optimality as possible.
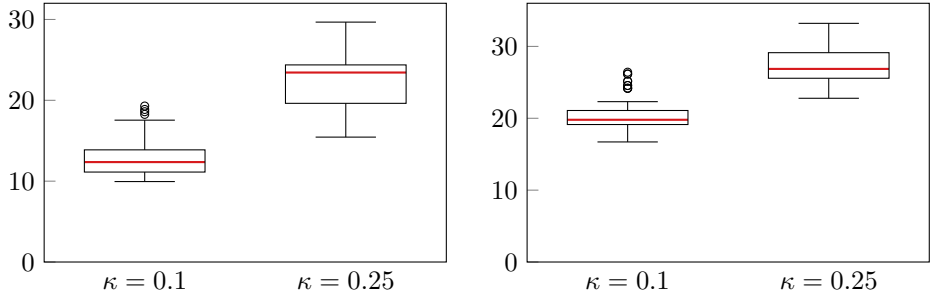
\begin{figure}
  \centering
\begin{tikzpicture}[baseline={(0,0)}]
\begin{axis}[
   width=69mm,
   height=50mm,
   xtick style={draw=none},
   xmin=1,
   xmax=2,
   ymin=0,
   ymax=32,
   xtick={1,2},
   xticklabels={$\kappa=0.1$,$\kappa=0.25$},
   enlarge x limits=.5,
   xtick pos=bottom,
   ytick pos=left,
   y tick label style={/pgf/number format/.cd,%
      scaled y ticks=false,
      set thousands separator={},
      fixed,
      precision=0},
   ]

   \draw[thick, my-red] (axis cs:0.7,12.358219795846773) -- (axis cs:1.3,12.358219795846773);
   \draw[thick, my-red] (axis cs:1.7,23.434145328156887) -- (axis cs:2.3,23.434145328156887);

   \draw (axis cs:0.7,11.12567828595021) rectangle (axis cs:1.3,13.871160434476554);
   \draw (axis cs:1.7,19.625222178336312) rectangle (axis cs:2.3,24.37746142324737);

   \draw (axis cs:1,9.9489542118515) -- (axis cs:1,11.12567828595021);
   \draw (axis cs:0.85,9.9489542118515) -- (axis cs:1.15,9.9489542118515);

   \draw (axis cs:2,15.447112609450924) -- (axis cs:2,19.625222178336312);
   \draw (axis cs:1.85,15.447112609450924) -- (axis cs:2.15,15.447112609450924);

   \draw (axis cs:1,13.871160434476554) -- (axis cs:1,17.543782702707446);
   \draw (axis cs:0.85,17.543782702707446) -- (axis cs:1.15,17.543782702707446);

   \draw (axis cs:2,24.37746142324737) -- (axis cs:2,29.670220988201514);
   \draw (axis cs:1.85,29.670220988201514) -- (axis cs:2.15,29.670220988201514);

   \addplot [only marks, mark=o, mark size=1.5pt] table {
      1 18.26079017047752
      1 18.86783552907769
      1 19.291262632824278
      1 18.587291497057635
      };
\end{axis}
\end{tikzpicture}
\quad
\begin{tikzpicture}[baseline={(0,0)}]
\begin{axis}[
   width=69mm,
   height=50mm,
   xtick style={draw=none},
   xmin=1,
   xmax=2,
   ymin=0,
   ymax=36,
   xtick={1,2},
   xticklabels={$\kappa=0.1$,$\kappa=0.25$},
   enlarge x limits=.5,
   xtick pos=bottom,
   ytick pos=left,
   y tick label style={/pgf/number format/.cd,%
      scaled y ticks=false,
      set thousands separator={},
      fixed,
      precision=0},
   ]

   \draw[thick, my-red] (axis cs:0.7,19.789900980293044) -- (axis cs:1.3,19.789900980293044);
   \draw[thick, my-red] (axis cs:1.7,26.866616277654003) -- (axis cs:2.3,26.866616277654003);

   \draw (axis cs:0.7,19.128273650585147) rectangle (axis cs:1.3,21.081441651623457);
   \draw (axis cs:1.7,25.566723271756068) rectangle (axis cs:2.3,29.125104927642223);

   \draw (axis cs:1,16.706621165879895) -- (axis cs:1,19.128273650585147);
   \draw (axis cs:0.85,16.706621165879895) -- (axis cs:1.15,16.706621165879895);

   \draw (axis cs:2,22.781977370500442) -- (axis cs:2,25.566723271756068);
   \draw (axis cs:1.85,22.781977370500442) -- (axis cs:2.15,22.781977370500442);

   \draw (axis cs:1,21.081441651623457) -- (axis cs:1,22.310383429638467);
   \draw (axis cs:0.85,22.310383429638467) -- (axis cs:1.15,22.310383429638467);

   \draw (axis cs:2,29.125104927642223) -- (axis cs:2,33.207080790093045);
   \draw (axis cs:1.85,33.207080790093045) -- (axis cs:2.15,33.207080790093045);

   \addplot [only marks, mark=o, mark size=1.5pt] table {
      1 26.155228662854
      1 26.41122360520129
      1 25.157609181822444
      1 24.55732852577625
      1 24.172292735750933
      1 26.155228662854
      1 25.157609181822444
      1 24.55732852577625
      1 24.172292735750933
      1 24.17601654350902
      };
\end{axis}
\end{tikzpicture}
  \caption{Boxplots for the gain of recovery (in \si{\percent}).
    Results for AP are shown on the left, whereas results for SSCFLP are
    shown on the right. Instances with the recovery parameter $k = \lceil
    \kappa n \rceil$ are considered, where $n$ is the number of variables in
    each instance.}
  \label{fig:gain-recovery}
\end{figure}
From Figure~\ref{fig:gain-recovery}, it can be clearly seen that incorporating
the possibility of recovery actions in the planning phase through a two-stage
robust framework enables more flexible and informed decision-making, leading to
substantially reduced costs. In particular, up to \SI{29.67}{\percent} and
\SI{33.21}{\percent} of the costs can be saved for AP and SSCFLP, respectively,
if recovery actions are incorporated in the modeling. Moreover, our results
suggest that these cost benefits increase with the number of allowed recovery
actions.
Nevertheless, these benefits also come at a certain price. We thus also analyze
the ``price of recovery'', i.e., the additional cost incurred by a recoverable
robust solution if recovery actions are not allowed after all.
To this end, we first compute an optimal first-stage decision~$x^*$ for the
$k$-delete  recoverable robust problem~\eqref{eq:recoverable-robust-problem}
and then evaluate the worst-case cost of this solution under the assumption
that no components of~$x^*$ can be deleted. 
We then compare this objective value to the one obtained by solving the robust
problem without recovery (i.e., Problem~\eqref{eq:recoverable-robust-problem}
with $k=0$).
For the price of recovery, we show boxplots in
Figure~\ref{fig:cost-recovery}. We observe that incorporating a recoverable
robust solution without being able to delete items afterward can increase costs
by up to \SI{4.49}{\percent} and \SI{15.07}{\percent} for AP and SSCFLP,
respectively. Comparing these increases with the potential savings mentioned
above shows that it can clearly be beneficial to incorporate recovery actions
in the decision-making process. Overall, our results  indicate that allowing
for recovery actions can yield substantial cost reductions, whereas the
associated cost increase without the possibility of recovery remains moderate
and manageable.

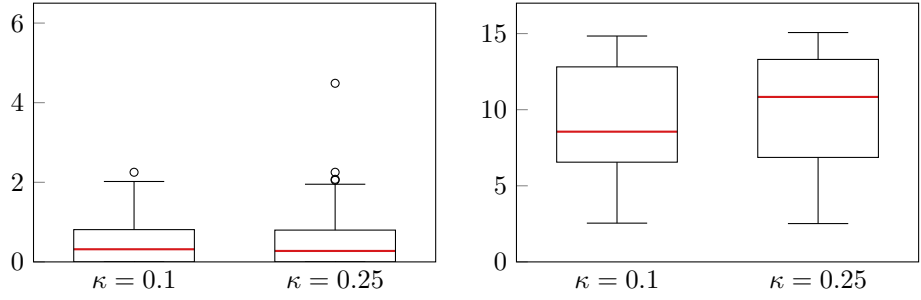
\begin{figure}
  \centering
\begin{tikzpicture}[baseline={(0,0)}]
\begin{axis}[
   width=69mm,
   height=50mm,
   xtick style={draw=none},
   xmin=1,
   xmax=2,
   ymin=0,
   ymax=6.5,
   xtick={1,2},
   xticklabels={$\kappa=0.1$,$\kappa=0.25$},
   enlarge x limits=.5,
   xtick pos=bottom,
   ytick pos=left,
   y tick label style={/pgf/number format/.cd,%
      scaled y ticks=false,
      set thousands separator={},
      fixed,
      precision=0},
   ]
   
   \draw[thick, my-red] (axis cs:0.7,0.3176161561416474) -- (axis cs:1.3,0.3176161561416474);
   \draw[thick, my-red] (axis cs:1.7,0.2751445740706333) -- (axis cs:2.3,0.2751445740706333);

   \draw (axis cs:0.7,0.0) rectangle (axis cs:1.3,0.8109712480003659);
   \draw (axis cs:1.7,0.0) rectangle (axis cs:2.3,0.798971182400113);

   \draw (axis cs:1,0.0) -- (axis cs:1,0.0);
   \draw (axis cs:0.85,0.0) -- (axis cs:1.15,0.0);

   \draw (axis cs:2,0.0) -- (axis cs:2,0.0);
   \draw (axis cs:1.85,0.0) -- (axis cs:2.15,0.0);

   \draw (axis cs:1,0.8109712480003659) -- (axis cs:1,2.02019691869465);
   \draw (axis cs:0.85,2.02019691869465) -- (axis cs:1.15,2.02019691869465);

   \draw (axis cs:2,0.798971182400113) -- (axis cs:2,1.9512099940975896);
   \draw (axis cs:1.85,1.9512099940975896) -- (axis cs:2.15,1.9512099940975896);

   \addplot [only marks, mark=o, mark size=1.5pt] table {
      1 2.252242107017536
      };

   \addplot [only marks, mark=o, mark size=1.5pt] table {
      2 2.252242107017536
      2 2.074680188049012
      2 2.051271531940862
      2 4.487165105240048
      };
\end{axis}
\end{tikzpicture}
\quad
\begin{tikzpicture}[baseline={(0,0)}]
\begin{axis}[
   width=69mm,
   height=50mm,
   xtick style={draw=none},
   xmin=1,
   xmax=2,
   ymin=0,
   ymax=17,
   xtick={1,2},
   xticklabels={$\kappa=0.1$,$\kappa=0.25$},
   enlarge x limits=.5,
   xtick pos=bottom,
   ytick pos=left,
   y tick label style={/pgf/number format/.cd,%
      scaled y ticks=false,
      set thousands separator={},
      fixed,
      precision=0},
   ]

   \draw[thick, my-red] (axis cs:0.7,8.556222320780423) -- (axis cs:1.3,8.556222320780423);
   \draw[thick, my-red] (axis cs:1.7,10.838628631917608) -- (axis cs:2.3,10.838628631917608);

   \draw (axis cs:0.7,6.552551364253379) rectangle (axis cs:1.3,12.813775275377676);
   \draw (axis cs:1.7,6.865260655115561) rectangle (axis cs:2.3,13.303193577162057);

   \draw (axis cs:1,2.546590731288029) -- (axis cs:1,6.552551364253379);
   \draw (axis cs:0.85,2.546590731288029) -- (axis cs:1.15,2.546590731288029);

   \draw (axis cs:2,2.5176881128102564) -- (axis cs:2,6.865260655115561);
   \draw (axis cs:1.85,2.5176881128102564) -- (axis cs:2.15,2.5176881128102564);

   \draw (axis cs:1,12.813775275377676) -- (axis cs:1,14.840842138600289);
   \draw (axis cs:0.85,14.840842138600289) -- (axis cs:1.15,14.840842138600289);

   \draw (axis cs:2,13.303193577162057) -- (axis cs:2,15.06933460963741);
   \draw (axis cs:1.85,15.06933460963741) -- (axis cs:2.15,15.06933460963741);
\end{axis}
\end{tikzpicture}
  \caption{Boxplots of the price of recovery (in \si{\percent}).
    Results for AP are shown on the left, whereas results for SSCFLP are
    shown on the right. Instances with the recovery parameter $k = \lceil
    \kappa n \rceil$ are considered, where $n$ is the number of variables in
    each instance.}
  \label{fig:cost-recovery}
\end{figure}

Let us now comment on the impact of robustness on the costs. To this
end, we compare the optimal objective value of a recoverable
robust solution to that of the respective deterministic two-stage variant of
the problem (i.e., Problem~\eqref{eq:recoverable-robust-problem}
with~$\Gamma=0$), in which the worst-case cost increase is computed ex post.
In Figure~\ref{fig:price-robust}, we show boxplots for the  cost savings
obtained through considering a robust framework. We observe that these savings
increase with larger values of~$\Gamma$. Overall, up to \SI{16.52}{\percent}
and \SI{21.74}{\percent} of the costs can be saved for AP and SSCFLP,
respectively. This demonstrates that our~$k$-delete recoverable 
robust framework allows for decisions that are more resilient against
uncertainties and cost variability.
Finally, we note that the previous analysis is closely related to the so-called
price of robustness \parencite{Bertsimas_Sim:2004}, which measures the increase
in the optimal objective function due to robustification. However, our metric
differs from the price of robustness in the sense that we evaluate the decrease
in the optimal objective function when comparing robust and nominal solutions
under adverse settings.

\begin{figure}
  \centering
\begin{tikzpicture}
\begin{axis}[
   width=69mm,
   height=55mm,
   xtick style={draw=none},
   xmin=1,
   xmax=3,
   ymin=0,
   ymax=18,
   xtick={1,2,3},
   xticklabels={$\gamma=0.1$,$\gamma=0.25$,$\gamma=0.5$},
   enlarge x limits=.33,
   xtick pos=bottom,
   ytick pos=left,
   y tick label style={/pgf/number format/.cd,%
      scaled y ticks=false,
      set thousands separator={},
      fixed,
      precision=0},
   ]

   \draw[thick, my-red] (axis cs:0.7,6.5542587007490045) -- (axis cs:1.3,6.5542587007490045);
   \draw[thick, my-red] (axis cs:1.7,9.906509254930668) -- (axis cs:2.3,9.906509254930668);
   \draw[thick, my-red] (axis cs:2.7,10.701419169143449) -- (axis cs:3.3,10.701419169143449);

   \draw (axis cs:0.7,5.324784502302727) rectangle (axis cs:1.3,7.358337986380054);
   \draw (axis cs:1.7,5.522845798037141) rectangle (axis cs:2.3,11.735065283681239);
   \draw (axis cs:2.7,6.983708318256094) rectangle (axis cs:3.3,14.123728421143914);

   \draw (axis cs:1,3.508759618387304) -- (axis cs:1,5.324784502302727);
   \draw (axis cs:0.85,3.508759618387304) -- (axis cs:1.15,3.508759618387304);

   \draw (axis cs:2,3.2573183800704237) -- (axis cs:2,5.522845798037141);
   \draw (axis cs:1.85,3.2573183800704237) -- (axis cs:2.15,3.2573183800704237);

   \draw (axis cs:3,4.76188586553228) -- (axis cs:3,6.983708318256094);
   \draw (axis cs:2.85,4.76188586553228) -- (axis cs:3.15,4.76188586553228);

   \draw (axis cs:1,7.358337986380054) -- (axis cs:1,9.278302689161395);
   \draw (axis cs:0.85,9.278302689161395) -- (axis cs:1.15,9.278302689161395);

   \draw (axis cs:2,11.735065283681239) -- (axis cs:2,14.953201153265638);
   \draw (axis cs:1.85,14.953201153265638) -- (axis cs:2.15,14.953201153265638);

   \draw (axis cs:3,14.123728421143914) -- (axis cs:3,16.52166729709871);
   \draw (axis cs:2.85,16.52166729709871) -- (axis cs:3.15,16.52166729709871);
\end{axis}
\end{tikzpicture}
\quad
\begin{tikzpicture}
\begin{axis}[
   width=69mm,
   height=55mm,
   xtick style={draw=none},
   xmin=1,
   xmax=3,
   ymin=0,
   ymax=23,
   xtick={1,2,3},
   xticklabels={$\gamma=0.1$,$\gamma=0.25$,$\gamma=0.5$},
   enlarge x limits=.33,
   xtick pos=bottom,
   ytick pos=left,
   y tick label style={/pgf/number format/.cd,%
      scaled y ticks=false,
      set thousands separator={},
      fixed,
      precision=0},
   ]

   \draw[thick, my-red] (axis cs:0.7,13.153090693348648) -- (axis cs:1.3,13.153090693348648);
   \draw[thick, my-red] (axis cs:1.7,15.648859437748328) -- (axis cs:2.3,15.648859437748328);
   \draw[thick, my-red] (axis cs:2.7,17.711327911975555) -- (axis cs:3.3,17.711327911975555);

   \draw (axis cs:0.7,11.372360474387307) rectangle (axis cs:1.3,14.021240365046296);
   \draw (axis cs:1.7,12.270637025418889) rectangle (axis cs:2.3,17.250365729400286);
   \draw (axis cs:2.7,13.942824312625264) rectangle (axis cs:3.3,18.877698385115217);

   \draw (axis cs:1,8.382369368510222) -- (axis cs:1,11.372360474387307);
   \draw (axis cs:0.85,8.382369368510222) -- (axis cs:1.15,8.382369368510222);

   \draw (axis cs:2,9.866168629751526) -- (axis cs:2,12.270637025418889);
   \draw (axis cs:1.85,9.866168629751526) -- (axis cs:2.15,9.866168629751526);

   \draw (axis cs:3,10.332813901690848) -- (axis cs:3,13.942824312625264);
   \draw (axis cs:2.85,10.332813901690848) -- (axis cs:3.15,10.332813901690848);

   \draw (axis cs:1,14.021240365046296) -- (axis cs:1,17.312564819234662);
   \draw (axis cs:0.85,17.312564819234662) -- (axis cs:1.15,17.312564819234662);

   \draw (axis cs:2,17.250365729400286) -- (axis cs:2,20.698596512629454);
   \draw (axis cs:1.85,20.698596512629454) -- (axis cs:2.15,20.698596512629454);

   \draw (axis cs:3,18.877698385115217) -- (axis cs:3,21.742113577948995);
   \draw (axis cs:2.85,21.742113577948995) -- (axis cs:3.15,21.742113577948995);
\end{axis}
\end{tikzpicture}
  \caption{Boxplots of the impact of robustness (in \si{\percent}).
    Results for AP are shown on the left, whereas results for SSCFLP are
    shown on the right. Instances with the robustness parameter $\Gamma =
    \lceil \gamma n \rceil$ are considered, where $n$ is the number of
    variables in each instance.}
  \label{fig:price-robust}
\end{figure}
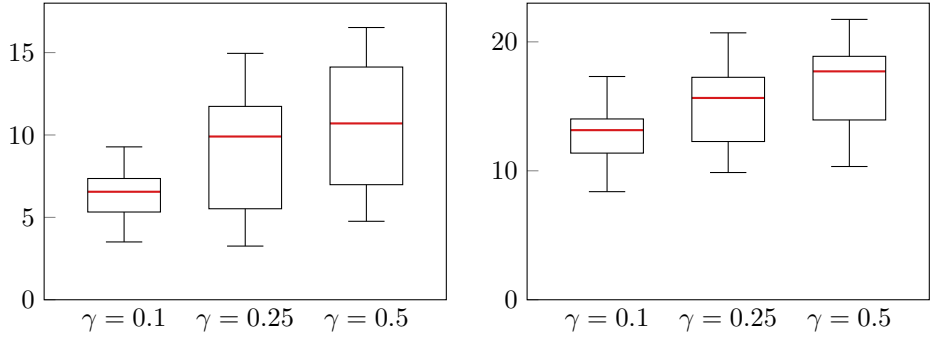


\section{Conclusion}
\label{sec:conclusion}

In this paper, we study the~$k$-delete recoverable robust $0$--$1$ problem in
which a decision-maker solves a combinatorial optimization problem subject to
objective uncertainty.
In this setting, the decision-maker first commits to an initial plan and may
revoke up to~$k$ components of this initial decision after the uncertainty is
revealed.
The underlying uncertainty is modeled using a budgeted uncertainty set so
that the decision-maker only hedges against at most~$\Gamma$ deviations in
the uncertain parameters.
We derive four equivalent reformulations of the~$k$-delete recoverable
robust problem---a scenario-based, a projection-based, an extended, and
a compact reformulation. For each formulation, we present exact solution
methods, including (i)~compact formulations that can be tackled using
general-purpose MILP solvers, (ii) branch-and-cut frameworks, and
(iii)~column-and-constraint generation algorithms.
The performance of all presented solution approaches is assessed in an
extensive computational study on instances of the assignment problem and the
single-source capacitated facility location problem.
Our computational results show that incorporating the possibility of recovery
actions in the planning phase can substantially improve the optimal objective
value compared to withdrawing commitments after solving a deterministic model.
Among all solution methods and reformulations that we consider in this paper,
we further observe that using a column-and-constraint-generation method to
solve the extended formulation seems to be the most effective one.

Finally, let us mention that our modeling of the~$k$-delete recoverable robust
problem allows recovered solutions to violate some constraints of the
original model. For instance, some customers may remain unserved in a facility
location problem.
Although such situations are common in practice, one may still be
interested in models that preserve feasibility after recovery. Developing
solution methods for such models typically requires different solution
techniques, which is beyond the scope of the present paper and left for future
research.
Another aspect that remains open is the computational complexity of the
$k$-delete recoverable robust problem studied in this work. Similar settings,
however, have been studied in \textcite{Gruene_Wulf:2024}, which may provide
useful insights for addressing this question in future work.
Moreover, it may be interesting to study alternative
recovery actions, such as allowing insertions or both insertions and
deletions. Extending the solution methods developed here, together with
existing approaches in recoverable robust optimization, to handle such recovery
actions in general combinatorial problems would provide a natural next step.
In addition, future work could consider other types of uncertainty, such as
uncertain constraints, and explore robust models beyond the budgeted
uncertainty framework studied in this work.


\section*{Acknowledgment}

This research was funded in part by the French National Research
Agency (ANR) under project ANR-24-CE48-3565-03.


\ifSubmission
  \section*{Conflict of Interest}
  The authors declare that they have no conflict of interest.
\fi

\printbibliography

\appendix
\markboth{APPENDIX}{APPENDIX}  
\section{Proofs of Lemmas~\ref{lem:proj-ref} and~\ref{lem:pos-neg-set}}
\label{sec:proofs}

In this section, we provide the omitted proofs of Lemmas~\ref{lem:proj-ref}
and~\ref{lem:pos-neg-set}, which are used in the proof of
Proposition~\ref{prop:proj-ref}.

\begin{proof}[\bf Proof of Lemma~\ref{lem:proj-ref}]
  Let~$x \in X$ and~$u \in \uncertaintySet_\Gamma$ be given arbitrarily. Then,
  we obtain
  \begin{equation*}
    \begin{split}
      \gamma(x,u)
      & = \max_y \Defset{\sum_{i=1}^n \left( \nomCost_i x_i +
        \devCost_i x_i u_i \right) y_i}{\sum_{i=1}^n y_i \leq k,\,
        y \in [0,1]^n}
      \\
      & = \min_{\nu,w \geq 0} \Defset{k\nu + \sum_{i=1}^n w_i}{\nu + w_i \geq
          \nomCost_i x_i + \devCost_i x_i u_i,\, i \in N}
      \\
      & = \min_{\nu \geq 0} \Set{k\nu + \sum_{i=1}^n \max \Set{0,\,
        \nomCost_i x_i + \devCost_i x_i u_i - \nu}}
      \\
      & = - \max_{\nu \geq 0} \Set{-k\nu - \sum_{i=1}^n \max \Set{0,\, \nomCost_i
        x_i + \devCost_i x_i u_i - \nu}}.
    \end{split}
  \end{equation*}
  Here, the first equality follows from the total unimodularity of the
  constraint matrix, whereas the second one is due to strong duality. The third
  equality follows from optimality arguments. Hence, we obtain
  \begin{align*}
    \beta(x) =
    \max_{u \in \uncertaintySet_\Gamma,\nu \geq 0}
    \Set{\sum_{i=1}^n \nomCost_ix_i + \sum_{i=1}^n \devCost_ix_iu_i - k\nu
    - \sum_{i=1}^n \max \Set{0,\, \nomCost_i x_i + \devCost_i x_i u_i - \nu}}.
  \end{align*}
  For all~$i \in N$, we have
  \begin{align*}
    & \ \max \Set{0,\, \nomCost_i x_i + \devCost_i x_i u_i - \nu}
    \\
    = & \ \max \Set{0,\, \nomCost_i - \nu} x_i(1 - u_i)
        + \max \Set{0,\, \nomCost_i + \devCost_i - \nu} x_i u_i
  \end{align*}
  and we further distinguish the following two cases.
  \begin{enumerate}
  \item If $\costFun{i}{\nu} < 0$, \ie, $i \in \negSet{\nu}$, we obtain
    \begin{equation*}
      \costFun{i}{\nu}
      = \min \Set{\devCost_i,\, \nu - \nomCost_i}
      = \nu - \nomCost_i
      = \min \Set{0,\, \nu - \nomCost_i}
    \end{equation*}
    because~$\devCost_i \geq 0$ holds.
  \item If $\costFun{i}{\nu} \geq 0$, \ie, $i \in \posSet{\nu}$, we
    obtain~$\nu \geq \nomCost_i$ and, thus,
    $\min \set{0,\, \nu - \nomCost_i} = 0$.
  \end{enumerate}
  Taking all previous considerations into account yields
  \begin{align*}
    & \ \sum_{i=1}^n \nomCost_ix_i + \sum_{i=1}^n \devCost_ix_iu_i - k\nu
      - \sum_{i=1}^n \max \Set{0,\, \nomCost_i x_i + \devCost_i x_i u_i - \nu}
    \\
    = & \ \sum_{i=1}^n \nomCost_i x_i
        + \sum_{i=1}^n \devCost_i x_i u_i - k\nu
        + \sum_{i=1}^n \min \Set{0,\, \nu - \nomCost_i} x_i(1 - u_i)
    \\
    & \ + \sum_{i=1}^n \min \Set{\devCost_i,\, \nu - \nomCost_i} x_iu_i
      - \sum_{i=1}^n \devCost_i x_iu_i
    \\
    = & \ \sum_{i=1}^n \nomCost_i x_i - k\nu
        + \sum_{i=1}^n \min \Set{0,\, \nu - \nomCost_i} x_i(1 - u_i)
        + \sum_{i=1}^n \costFun{i}{\nu} x_i u_i
    \\
    = & \ \sum_{i=1}^n \nomCost_i x_i - k\nu
        + \sum_{i \in \negSet{\nu}} \min \Set{0,\, \nu - \nomCost_i}x_i(1 -
         u_i)
    \\
    & \ + \sum_{i \in \posSet{\nu}} \min \Set{0,\, \nu - \nomCost_i}x_i(1
      - u_i) + \sum_{i=1}^n \costFun{i}{\nu} x_iu_i
    \\
    = & \ \sum_{i=1}^n \nomCost_i x_i - k\nu
        + \sum_{i \in \negSet{\nu}} \costFun{i}{\nu}x_i(1 - u_i)
        + \sum_{i=1}^n \costFun{i}{\nu} x_iu_i
    \\
    = & \ \sum_{i=1}^n \nomCost_i x_i - k\nu
        + \sum_{i \in \negSet{\nu}} \costFun{i}{\nu} x_i
        + \sum_{i \in \posSet{\nu}} \costFun{i}{\nu} x_i u_i,
  \end{align*}
  which concludes the proof.
\end{proof}

\begin{proof}[\bf Proof of Lemma~\ref{lem:pos-neg-set}]
  \begin{enumerate}
  \item The claim immediately follows from the definition of the respective
    sets.
  \item Let~$i \in \leSet{\nu^*}$ be given arbitrarily.
    By construction, we then have~$\devCost_i > \nu^* - \nomCost_i \geq
    \underline{\nu} - \nomCost_i$, which yields
    \begin{align*}
      \costFun{i}{\underline{\nu}} + (\nu^* - \underline{\nu})
      & = \min \Set{\devCost_i,\, \underline{\nu} - \nomCost_i} + (\nu^* -
        \underline{\nu})
        = (\underline{\nu} - \nomCost_i) + (\nu^* - \underline{\nu})
      \\
      & = \nu^* - \nomCost_i
        = \min \Set{\devCost_i,\, \nu^* - \nomCost_i}
      \\
      & = \costFun{i}{\nu^*}.
    \end{align*}
    This proves the first equality in~(ii).
    We now show the second one.
    To this end, let us first assume that~$\bar{\nu} > \nomCost_i +
    \devCost_i$ holds. Because~$i \in \leSet{\nu^*}$,
    we then obtain~$\bar{\nu} > \nomCost_i + \devCost_i > \nu^*$,
    which contradicts the definition of~$\bar{\nu}$.
    Hence, $\bar{\nu} \leq \nomCost_i + \devCost_i$ has to hold and we obtain
    \begin{align*}
      \costFun{i}{\nu^*}
      & = \min \Set{\devCost_i,\, \nu^* - \nomCost_i}
        = \nu^* - \nomCost_i
      \\
      & = (\bar{\nu} - \nomCost_i) - (\bar{\nu} - \nu^*)
      = \min \Set{\devCost_i,\, \bar{\nu} - \nomCost_i} -
        (\bar{\nu} - \nu^*)
      \\
      & = \costFun{i}{\bar{\nu}} - (\bar{\nu} - \nu^*).
    \end{align*}
  \item Let~$i \in \geqSet{\nu^*}$ be given arbitrarily, i.e.,
    $\nu^* \geq \nomCost_i + \devCost_i$.
    By construction, we have~$\underline{\nu} \geq \nomCost_i + \devCost_i$
    as well as
    \begin{equation*}
      \costFun{i}{\underline{\nu}}
      = \min \Set{\devCost_i,\, \underline{\nu} - \nomCost_i}
      = \devCost_i
      = \min \Set{\devCost_i,\, \nu^* - \nomCost_i}
      = \costFun{i}{\nu^*}.
    \end{equation*}
    Moreover, we have~$\bar{\nu} \geq \nu^* \geq \nomCost_i + \devCost_i$,
    which yields
    \begin{equation*}
      \costFun{i}{\bar{\nu}} = \min \Set{\devCost_i,\, \bar{\nu} - \nomCost_i}
      = \devCost_i = \min \Set{\devCost_i,\, \nu^* - \nomCost_i}
      = \costFun{i}{\nu^*}.
    \end{equation*}
  \item From~$i \in \negSet{\nu^*} \cap \posSet{\bar{\nu}}$, we obtain
    $\bar{\nu} \geq \nomCost_i > \nu^*$ and, thus, $\bar{\nu} = \nomCost_i$
    has to hold by the definition of~$\bar{\nu}$. As a result, we have
    $\costFun{i}{\bar{\nu}} = \min \set{\devCost_i,\, \bar{\nu} - \nomCost_i} = 
    0$ because of~$\devCost_i \geq 0$.
  \item We prove the claim by contradiction.
    Hence, suppose that~$\nu^* \notin \newUncertaintySet$ holds and that there
    exists~$i \in \posSet{\nu^*} \cap \negSet{\underline{\nu}}$.
    From~$i \in \posSet{\nu^*}$, $\nu^* \notin \newUncertaintySet$,
    and~$\nomCost_i \in \newUncertaintySet$, we obtain~$\nu^* > \nomCost_i$.
    Because~$i \in \negSet{\underline{\nu}}$, we further
    have~$\nomCost_i > \underline{\nu}$.
    Overall, this yields~$\nu^* > \nomCost_i > \underline{\nu}$, which is a
    contradiction to \mbox{$\underline{\nu} = \max\, \defset{\nu}{\nu \in
        \newUncertaintySet,\, \nu \leq \nu^*} \geq \nomCost_i$}. \qedhere
  \end{enumerate}
\end{proof}

\section{Proof of Theorem~\ref{thm:extended-formulation}}
\label{sec:proof-thm-ext-form}

By Proposition~\ref{prop:proj-ref},
Problem~\eqref{eq:recoverable-robust-problem} can be solved as
\begin{align*}
  \min_{x,\eta} \quad & \initCost^\top x + \eta
  \\
  \st \quad & \eta \geq \alpha(x,\nu) - k\nu, \quad \nu \in
              \newUncertaintySet,
  \\
                      & x \in X,\, \eta \in \R_{\geq 0}.
\end{align*}
Let now~$x \in X$ and~$\nu \in \newUncertaintySet$ be given
arbitrarily. Because~$\alpha(x,\nu)$ corresponds to a classic discrete
budgeted uncertainty model, we obtain
\begin{align*}
  \alpha(x,\nu)
  = & \ \max_u \Defset{ \sum_{i=1}^n \nomCostFun{i}{\nu} x_i
    + \sum_{i=1}^n \devCostFun{i}{\nu} x_i u_i}{
    \sum_{i=1}^n u_i \leq \Gamma,\, u \in [0,1]^n}
  \\
  = & \ \sum_{i=1}^n \nomCostFun{i}{\nu} x_i
      + \min_{w^\nu,z^\nu \geq 0} \Defset{\Gamma w^\nu + \sum_{i=1}^n
      z^\nu_i}{w^\nu + z^\nu_i \geq \devCostFun{i}{\nu} x_i,\, i \in N}.
\end{align*}
Here, the first equality follows from the total unimodularity of the
constraint matrix (cf.\ Proposition~3.3 in \textcite{Wolsey:2020}), whereas
the second one is due to strong duality. Overall, we thus obtain
\begin{align*}
  \eta
  \geq \sum_{i=1}^n \nomCostFun{i}{\nu} x_i
  + \min_{w^\nu,z^\nu \geq 0} \Defset{\Gamma w^\nu + \sum_{i=1}^n
  z^\nu_i}{w^\nu + z^\nu_i \geq \devCostFun{i}{\nu} x_i,\, i \in N}
  - k\nu.
\end{align*}
If, for fixed~$x \in X$ and~$\nu \in \newUncertaintySet$, there is a
pair~$(w^\nu,z^\nu)$ that is feasible for the problem
\begin{equation}
  \label{eq:extended-form-prob}
  \min_{w^\nu,z^\nu \geq 0} \quad
  \Gamma w^\nu + \sum_{i=1}^n z^\nu_i
  \quad \st \quad w^\nu + z^\nu_i \geq \devCostFun{i}{\nu} x_i,
  \quad i \in N,
\end{equation}
and that satisfies the inequality
\begin{equation*}
  \eta \geq \sum_{i=1}^n \nomCostFun{i}{\nu} x_i + \Gamma w^\nu +
  \sum_{i=1}^n z^\nu_i - k\nu,
\end{equation*}
the latter is particularly satisfied for an optimal solution to
Problem~\eqref{eq:extended-form-prob}. This concludes the proof.
\qed


\section{Proof of Proposition~\ref{prop:compact-formulation}}
\label{sec:proof-bs}

By Theorem~\ref{thm:extended-formulation},
Problem~\eqref{eq:recoverable-robust-problem} can be solved as
\begin{equation*}
  \begin{split}
    \min_{x,\eta} \quad & \initCost^\top x + \eta
    \\
    \st \quad & \eta \geq \sum_{i=1}^n \nomCostFun{i}{\nu} x_i
                + \varphi(x,\nu) - k\nu,
                \quad \nu \in \newUncertaintySet,
    \\
                        & x \in X,\, \eta \in \R_{\geq 0},
  \end{split}
\end{equation*}
where, for fixed~$x \in X$ and~$\nu \in \newUncertaintySet$, we have
\begin{equation*}
  \varphi(x,\nu) \define \min_{w^\nu, z^\nu \geq 0} \Defset{
    \Gamma w^\nu + \sum_{i=1}^n z^\nu_i}{w^\nu + z^\nu_i \geq
    \devCostFun{i}{\nu} x_i,\, i \in N}.
\end{equation*}
Let now~$x \in X$ and~$\nu \in \newUncertaintySet$ be given arbitrarily.
By the definition of the perturbation function~$\sigma_\nu$, applying
the same steps as in the proof of Theorem~3 in 
\textcite{Bertsimas_Sim:2003} and the proofs of the results in
\textcite{Lee_Kwon:2014} then yields
\begin{equation*}
  \varphi(x,\nu)
  = \min_{\ell \in \mathcal{L}}
  \Set{\Gamma \devCostFun{\sigma_\nu(\ell)}{\nu}
    + \sum_{i=1}^{\ell}
    \left( \devCostFun{\sigma_\nu(i)}{\nu} -
      \devCostFun{\sigma_\nu(\ell)}{\nu} \right) x_{\sigma_\nu(i)}}.
\end{equation*}
Hence, for all~$\nu \in \newUncertaintySet$, we obtain
\begin{align*}
  \eta & \geq \sum_{i=1}^n \nomCostFun{i}{\nu} x_i + \varphi(x,\nu) - k\nu
  \\
       & = \min_{\ell \in \mathcal{L}} \Set{
         \Gamma \devCostFun{\sigma_\nu(\ell)}{\nu} - k\nu
         + \sum_{i=1}^n \nomCostFun{i}{\nu} x_i
         + \sum_{i=1}^{\ell} \left( \devCostFun{\sigma_\nu(i)}{\nu} -
         \devCostFun{\sigma_\nu(\ell)}{\nu} \right) x_{\sigma_\nu(i)}}
  \\
       & =  \min_{\ell \in \mathcal{L}} \Set{\kappa_\ell(\nu)
         + \sum_{i=1}^n \compCostFun{\sigma_\nu(i)}{\nu}{\ell}
         x_{\sigma_\nu(i)}}. \tag*{\qed}
\end{align*}


\section{Proof of Proposition~\ref{prop:adversarial-prob-ref}}
\label{sec:proof-adversarial-prob}

Per definition, we have~$\scenarioCost{s}_i = \nomCost_i + \devCost_i u^s_i$
for all~$i \in N$ and~$s \in S$ with~$u^s \in \uncertaintySet_\Gamma$.
Hence, Problem~\eqref{eq:adversarial-prob} can
equivalently be stated as the max-min problem
\begin{equation*}
  \max_{z \in Z} \ \min_{y \in Y(x)} \quad
  \sum_{i=1}^n \left( \nomCost_i + \devCost_i z_i \right) y_i
\end{equation*}
with
\begin{equation*}
  Z \define \Defset{z \in \Set{0,1}^n}{\sum_{i=1}^n z_i \leq \Gamma}
\end{equation*}
and
\begin{equation*}
  Y(x) \define \Defset{y \in \Set{0,1}^n}{\sum_{i=1}^n y_i \geq
    \sum_{i=1}^n x_i - k,\, x \geq y}.
\end{equation*}
Applying Proposition~3.3 in \textcite{Wolsey:2020}, we can relax the
integrality restrictions~$y \in \set{0,1}^n$, i.e., we consider
\begin{equation}
  \label{eq:proof:adversarial-prob}
    \min_{y \in [0,1]^n} \quad
    \sum_{i=1}^n \left( \nomCost_i + \devCost_i z_i \right) y_i
    \quad \st \quad
    \sum_{i=1}^n y_i \geq \sum_{i=1}^n x_i - k,\,
    x \geq y
\end{equation}
as the inner minimization problem.
Problem~\eqref{eq:proof:adversarial-prob} is a continuous linear optimization
problem for fixed~$x \in X \subseteq \set{0,1}^n$ and the constraints~$y_i
\leq 1$ are redundant because of~$x \geq y$. Hence, the dual of
Problem~\eqref{eq:proof:adversarial-prob} is given by
\begin{equation*}
  \begin{split}
    \max_{\lambda,\mu} \quad
    & \left( \sum_{i=1}^n x_i - k\right) \lambda - \sum_{i=1}^n x_i\mu_i
    \\
    \st \quad & \lambda - \mu_i \leq \nomCost_i + \devCost_i z_i,
      \quad i \in N,
    \\
    & \mu \in \R^n_{\geq 0},\, \lambda \in \R_{\geq 0}.
  \end{split}
\end{equation*}
Finally, the claim follows by applying strong duality. \qed


\section{Supplementary Tables}
\label{sec:appendix-comp-results}

\begin{table}[h]
  \centering
  \caption{The number of instances solved to global optimality (``solved''; out
    of $240$ and $222$ instances for AP and SSCFLP,
    respectively), the number of instances for which a feasible point with
    finite but non-zero gap is found (``open gap''), and the number of
    instances with infinite gap (``infinite gap'') for the branch-and-cut
    approach applied to the projection-based formulation. For those instances
    with finite but non-zero gap, also the average gap (``average gap''; in
    \si{\percent}) is shown.}
  \begin{tabular}{llrrrr}
    \toprule
    problem & cut strategy & solved & open gap & average gap & infinite gap \\
    \midrule
    AP &
    \textsf{All-In} & 178 & 61 & 6.19 & 1 \\
    & \textsf{First-In} & 218 & 21 & 3.04 & 1 \\
    & \textsf{Shuffle-First-In} & 210 & 30 & 5.19 & 0 \\
    & \textsf{Most-Violated} & 211 & 29 & 6.36 & 0 \\
    \midrule
    SSCFLP &
    \textsf{All-In} & 109 & 75 & 6.31 & 38 \\
    & \textsf{First-In} & 194 & 7 & 3.90 & 21 \\
    & \textsf{Shuffle-First-In} & 196 & 8 & 1.87 & 18 \\
    & \textsf{Most-Violated} & 187 & 5 & 16.93 & 30 \\
    \bottomrule
  \end{tabular}
  \label{tab:projected-cuts}
\end{table}

\begin{table}[h]
  \centering
  \caption{The number of instances solved to global optimality (``solved''; out
    of $240$ and $222$ instances for AP and SSCFLP,
    respectively), the number of instances for which a
    feasible point with finite but non-zero gap is found (``open gap''), and
    the number of instances with infinite gap (``infinite gap'') for the
    approaches applied to the extended formulation.
    For those instances with finite but non-zero gap, also the average
    gap (``average gap''; in \si{\percent}) is shown.}
  \begin{tabular}{llrrrr}
    \toprule
    problem & approach & solved & open gap & average gap & infinite gap \\
    \midrule
    AP &
    MILP & 240 & 0 & -- & 0 \\
    & BnC & 232 & 8 & 1.17 & 0 \\
    & CCG & 240 & 0 & -- & 0 \\
    \midrule
    SSCFLP &
    MILP & 139 & 9 & 1.17 & 74 \\
    & BnC & 165 & 4 & 0.84 & 53 \\
    & CCG & 195 & 27 & 14.35 & 0 \\
    \bottomrule
  \end{tabular}
  \label{tab:extended}
\end{table}

\begin{table}[h]
  \centering
  \caption{The number of instances solved to global optimality (``solved''; out
    of $240$ and $222$ instances for AP and SSCFLP,
    respectively), the number of instances for which a
    feasible point with finite but non-zero gap is found (``open gap''), and
    the number of instances with infinite gap (``infinite gap'') for the
    approaches applied to the compact formulation.
    For those instances with finite but non-zero gap, also the average
    gap (``average gap''; in \si{\percent}) is shown.}
  \begin{tabular}{llrrrr}
    \toprule
    problem & approach & solved & open gap & average gap & infinite gap \\
    \midrule
    AP & 
    MILP & 58 & 2 & 11.10 & 180 \\
    & BnC & 181 & 32 & 21.62 & 27 \\
    & CCG & 238 & 0 & -- & 2 \\
    \midrule
    SSCFLP &
    MILP & 0 & 0 & -- & 222 \\
    & BnC & 100 & 72 & 10.86 & 50 \\
    & CCG & 0 & 0 & -- & 222 \\
    \bottomrule
  \end{tabular}
  \label{tab:compact}
\end{table}


\end{document}